\documentclass[12pt,leqno]{amsart}
\usepackage{latexsym,amsthm,amsmath,amssymb,mathrsfs,mathtools}
\usepackage{graphicx,caption}
\usepackage{tikz}
\usetikzlibrary{arrows,calc,patterns,cd}

\title{Topologically nontrivial counterexamples to Sard's theorem}

\author[P. Goldstein]{Pawe\l{}  Goldstein}
\address{Pawe\l{} Goldstein, Institute of Mathematics, Faculty of Mathematics, Informatics and Mechanics, University of Warsaw, Banacha 2, 02-097 Warsaw, Poland} \email{P.Goldstein@mimuw.edu.pl}
\thanks{P.G. was partially supported by National Science Center grant no 2012/05/E/ST1/03232.}
\author[P. Haj\l{}asz]{Piotr Haj\l{}asz}
\address{Piotr Hajlasz, Department of Mathematics, University of Pittsburgh, Pittsburgh, PA 15260, USA}
\email{hajlasz@pitt.edu}
\thanks{P.H.\ was supported by NSF grant DMS-1800457.}
\author[P. Pankka]{Pekka Pankka}
\address{Pekka Pankka, Department of Mathematics and Statistics, P.O. Box 68 (Gustaf H\"allstr\"omin katu 2b), FI-00014 University of Helsinki, Finland}
\email{pekka.pankka@helsinki.fi}
\thanks{P.P.\ was supported by the Academy of Finland project \#297258.}

\newlength{\currentparindent}

plus 6pt minus 12pt
\def\rank{{\rm rank\,}}

\def\eps{\varepsilon}

\def\id{{\rm id\, }}

\def\i{\mathfrak{i}}

\def\M{{\mathcal M}}
\def\N{{\mathcal N}}
\def\H{{\mathcal H}}
\def\HH{{\mathscr H}}

\newtheorem{theorem}{Theorem}
\newtheorem{lemma}[theorem]{Lemma}
\newtheorem{corollary}[theorem]{Corollary}

\theoremstyle{definition}
\newtheorem{remark}[theorem]{Remark}

\newcommand{\barint}{
\rule[.036in]{.12in}{.009in}\kern-.16in \displaystyle\int }

\newcommand{\barcal}{\mbox{$ \rule[.036in]{.11in}{.007in}\kern-.128in\int $}}
\newcommand{\overbar}[1]{\mkern 1.7mu\overline{\mkern-1.7mu#1\mkern-1.5mu}\mkern 1.5mu}

\newcommand{\Sph}{\mathbb S}

\def\tB{\tilde{B}}
\newcommand{\bbbp}{\mathbb P}
\newcommand{\bbbb}{\mathbb B}
\newcommand{\bbbz}{\mathbb Z}

\newcommand{\bbbr}{\mathbb R}

\def\mvint_#1{\mathchoice
          {\mathop{\vrule width 6pt height 3 pt depth -2.5pt
                  \kern -8pt \intop}\nolimits_{\kern -3pt #1}}%
          {\mathop{\vrule width 5pt height 3 pt depth -2.6pt
                  \kern -6pt \intop}\nolimits_{#1}}%
          {\mathop{\vrule width 5pt height 3 pt depth -2.6pt
                  \kern -6pt \intop}\nolimits_{#1}}%
          {\mathop{\vrule width 5pt height 3 pt depth -2.6pt
                  \kern -6pt \intop}\nolimits_{#1}}}

\numberwithin{theorem}{section} \numberwithin{equation}{section}

\begin{document}

\subjclass[2010]{Primary: 58C25; Secondary: 55Q25, 55Q40, 58A12}
\keywords{Sard theorem, Freudenthal suspension theorem, Hopf invariant}
\sloppy


\sloppy

\begin{abstract}
We prove the following dichotomy:
if $n=2,3$ and $f\in C^1(\Sph^{n+1},\Sph^n)$ is not homotopic to a constant map,
then there is an open set $\Omega\subset\Sph^{n+1}$ such that $\rank df=n$ on $\Omega$ and $f(\Omega)$ is dense in $\Sph^n$, while
for any $n\geq 4$, there is a map $f\in C^1(\Sph^{n+1},\Sph^n)$ that is not homotopic to a constant map and such that
$\rank df<n$ everywhere. The result in the case $n\geq 4$ answers a question of Larry Guth.
\end{abstract}

\maketitle

\section{Introduction}

In 1942, Sard \cite{sard} proved that if $f\in C^k(\bbbr^m,\bbbr^n)$ for $k>\max\{m-n,0\}$, then the set of critical values of $f$ has measure zero; see also \cite{sternberg}. In particular, if $\M^m$ and $\N^n$ are closed manifolds of dimensions $m\geq n$, respectively, $k>m-n$, and $f\in C^k(\M^m,\N^n)$ is a surjective map, then there is an open set
$\Omega\subset\M^m$ such that $\rank df=n$ everywhere in $\Omega$ and $f(\Omega)$ is dense in $\N^n$.

Sard's theorem is no longer true for $k\leq \max\{m-n,0\}$. In fact, Kaufman showed in \cite{kaufman} that, for each $n\geq 2$,
there exists a surjective map $f\in C^1([0,1]^{n+1},[0,1]^n)$ with $\rank df\leq 1$ {\em everywhere}.
However, Kaufman's mapping is a limit of a uniformly convergent sequence of
mappings into finite, one dimensional, piecewise linear trees, so it is topologically trivial in the sense that mimicking Kaufman's construction
in the case of $C^1$ mappings between spheres $\Sph^{n+1}$ and $\Sph^{n}$ would result in a mapping
$f\in C^{1}(\Sph^{n+1},\Sph^{n})$ with $\rank f\leq 1$ everywhere that is homotopic to a constant map.
Indeed, mappings into trees are contractible and so is their limit.

Since the homotopy groups $\pi_{n+1}(\Sph^n)\neq 0$ are non-trivial for $n\geq 2$ (see e.g.\;\cite{hatcher}),
one may ask whether it is possible to construct a Kaufman type example that is not homotopic to a constant map.

A mapping $f\in C^2(\Sph^{n+1},\Sph^n)$ that is not homotopic to a constant mapping is surjective and hence, according to Sard's theorem, there is an open set
$\Omega\subset\Sph^{n+1}$ having the property that
\begin{equation}
\label{eq1}
\text{$\rank df=n$ everywhere in $\Omega$ and $f(\Omega)$ is dense in $\Sph^n$.}
\end{equation}

In particular, there is no mapping $f\in C^2(\Sph^{n+1}, \Sph^n)$, satisfying $\rank df < n$ everywhere, that is not homotopic to a constant map. Note that the condition $\rank df <n$ is much weaker than Kaufman's $\rank df\le 1$.
This leads to two natural questions.

\noindent {\sc Question 1.}
{\em Is it possible to construct a mapping $f\in C^1(\Sph^{n+1},\Sph^n)$, $n\geq 2$, such that $\rank df<n$ everywhere and $f$ is not homotopic to a constant one?}

\noindent {\sc Question 2.}
{\em Let $n\geq 2$ and $1\leq m<n$ be given.
Is it possible to construct a mapping $f\in C^1(\Sph^{n+1},\Sph^n)$ such that $\rank df\leq m$ everywhere and $f$ is not homotopic to a constant one?}

We note in passing that, in the more general context of $C^1$-mappings between closed manifolds, it is easy to give examples of mappings and manifolds answering both questions. Consider, for example, the smooth map $f:\Sph^1\times\Sph^1\times\Sph^1\to\Sph^1\times\Sph^1$, $(x,y,z) \mapsto (x,y_o)$, where $y_o \in \Sph^1$, which is not homotopic to a constant map but satisfies $\rank df =1$ everywhere.

The questions stated above are essentially due to Larry Guth \cite[p.\ 1889]{guth}, who asked:
{\em We don't know any homotopically non-trivial $C^1$ maps from $\Sph^m$ to $\Sph^n$ with $\rank <n$. Does one exist?}
Guth \cite[Main Theorem]{guth} obtained a partial answer to Question~2 by showing a lower bound for the rank of the derivative of
homotopically non-trivial maps.

\begin{theorem}[Guth]
\label{ThGuth}
If $n\geq 2$ and $f\in C^1(\Sph^{n+1},\Sph^n)$ satisfies $\rank df<\left[\frac{n+2}{2}\right]$, then $f$ is homotopic to a constant map.
\end{theorem}
Here $[x]$ stands for the integer part of $x$.
In particular, the following maps are necessarily homotopic to constant maps:
\begin{equation}
\label{aaa1}
f\in C^1(\Sph^3,\Sph^2),
\quad
\text{or}
\quad
f\in C^1(\Sph^4,\Sph^3),
\quad
\text{with $\rank df<2$,}
\end{equation}
$$
f\in C^1(\Sph^5,\Sph^4),
\quad
\text{or}
\quad
f\in C^1(\Sph^6,\Sph^5),
\quad
\text{with $\rank df<3$,}
$$
$$
f\in C^1(\Sph^7,\Sph^6),
\quad
\text{or}
\quad
f\in C^1(\Sph^8,\Sph^7),
\quad
\text{with $\rank df<4$.}
$$
On the other hand, in the case of mappings $f\in C^1(\Sph^4,\Sph^3)$, Guth proved a stronger result than that in \eqref{aaa1}. Namely he proved in
\cite[Proposition~13.4]{guth} that if $f\in C^1(\Sph^4,\Sph^3)$ and
and $\rank df<3$, then $f$ is homotopic to a constant map. This led him to the following conjecture:

\noindent {\sc Conjecture~1.}  (Guth)
{\em Let $n\geq 5$ be odd. If $f\in C^1(\Sph^{n+1},\Sph^n)$ and $\rank df<\left[\frac{n+3}{2}\right]$, then $f$ is homotopic to a constant map.}

Note that if $n\geq 2$ is even, then the above claim is true by Theorem~\ref{ThGuth}, and when $n=3$ it is true by \cite[Proposition~13.4]{guth}, but it is
an open problem when $n\geq 5$ is odd. Guth also conjectured that the above estimate for the rank is sharp:

\noindent {\sc Conjecture~2.} (Guth)
{\em If $n\geq 4$, then there is a map $f\in C^1(\Sph^{n+1},\Sph^n)$ with $\rank df\leq \left[\frac{n+3}{2}\right]$ that is not homotopic to a constant map.}

Note that the if $n=2,3$, then the above claim is obvious, because $ \left[\frac{n+3}{2}\right]=n$ so any map satisfies the rank condition.

The aim of this paper is to prove the following result.
\begin{theorem}
\label{main}
For $n=2,3$ and each map $f\in C^1(\Sph^{n+1},\Sph^n)$ not homotopic to a constant map, there is an open set $\Omega\subset\Sph^{n+1}$ such that $\rank df=n$ on $\Omega$ and $f(\Omega)$ is dense in $\Sph^n$. In contrast, for each $n\ge 4$, there is a map $f\in C^1(\Sph^{n+1},\Sph^n)$ that is not homotopic to a constant map and such that $\rank df<n$ everywhere.
\end{theorem}
The case $n=2$ is relatively easy, it follows from Theorem~\ref{ThGuth}, but, in fact, it has been known before, see comments to Theorem~\ref{fromHST}.
The known proofs are based on estimates of the Hopf invariant. For the sake of completeness we provide a variant of such a proof.
The case $n=3$ was proved in \cite[Proposition~13.4]{guth} with a difficult argument based on the Steenrod squares. We provide a very different, and a more elementary
proof based on a generalized Hopf invariant introduced in \cite{HST}.
However, the case $n\geq 4$ is new. It answers Question~1 and Conjecture~2 for $n=4$ in the affirmative.

Modifying our proof slightly, we could show that if $f\in C^1(\Sph^{n+1},\Sph^n)$, $n\geq 3$, and $\rank df<3$, then $f$ is homotopic to a constant map. This is consistent with the estimates obtained by Guth when $n\leq 5$. However, in higher dimensions Theorem~\ref{ThGuth} gives a better estimate.

Note that when $n=2$ or $n=3$ and $f\in C^1(\Sph^{n+1},\Sph^n)$ is not homotopic to a constant map, then the conclusion \eqref{eq1} of Sard's theorem is still true
despite the fact that the mapping $f$ has less regularity than required in Sard's theorem.

We find it somewhat surprising that the situation changes at the dimension $n=4$. For example,
$\pi_4(\Sph^3)=\pi_5(\Sph^4)=\bbbz_2$, so the homotopy groups of the spheres are the same when $n=3$ and $n=4$, but the claim of Theorem~\ref{main}
is different in these dimensions.

The map constructed in the proof of Theorem~\ref{main} in the case $n\geq 4$ has $\rank df=n-1$ on a set of positive measure.
We do not know whether there exists a map $f\in C^1(\Sph^{n+1},\Sph^n)$ which is not homotopic to a constant map and satisfies $\rank df \leq n-2$ everywhere.
Looking for such a map would be a first step towards answering Conjecture~2.

The first part of Theorem~\ref{main} is a consequence of a slightly stronger result:

\begin{theorem}
\label{main2}
If $n=2,3$ and $f:\Sph^{n+1}\to \Sph^n$ is Lipschitz continuous and not homotopic to a constant map,
then there is a set $A\subset\Sph^{n+1}$ of positive measure such that $f$ is differentiable at every point of $A$,
$\rank df=n$ on $A$ and $f(A)$ is dense in $\Sph^n$.
\end{theorem}

If now $f\in C^1(\Sph^{n+1},\Sph^n)$ and $\rank df=n$ on $A$, then  $\rank df=n$ on an open set that contains $A$, so the first part of
Theorem~\ref{main} follows.

For $n=2$, Theorem~\ref{main2} follows from the following more general result related to the Hopf invariant.
This result is known and it follows from the so called Hopf invariant inequality, see \cite[Section~3.6]{Gro96}, \cite[pp.\ 358-359]{Gro07},
\cite[p.\ 1805 and p. 1818]{guth}, \cite{riviere}.
\begin{theorem}
\label{fromHST}
If $f:\Sph^{4n-1}\to\Sph^{2n}$ is Lipschitz and $\H f\neq 0$, then there is a set $A\subset\Sph^{4n-1}$ of positive measure such that
$f$ is differentiable at every point of $A$,
$\rank df=2n$ on $A$ and $f(A)$ is a dense subset of $\Sph^{2n}$.
\end{theorem}
Our proof is similar to the other known proofs. We decided to include details since they play an important role in the proof of Corollary~\ref{fromHST2}.

Since a map  $f:\Sph^{3}\to \Sph^2$ is not homotopic to a constant map if and only if the Hopf invariant $\H f\neq 0$ is non-trivial
(see Remark~\ref{r1}), we readily obtain that Theorem~\ref{fromHST} yields Theorem \ref{main2} in this case.  We note, again in passing,
that for each $n\geq 2$ there are mappings $\Sph^{4n-1}\to\Sph^{2n}$ which have a trivial Hopf invariant but which are not homotopic to a constant map.

The proof of Theorem~\ref{fromHST} (and hence the proofs of Theorems~\ref{main} and~\ref{main2} when $n=2$) is based on a
generalized Hopf invariant defined and studied in \cite{HST}.
This is a non-standard generalization that requires the use of the $L^p$-Hodge decomposition.
The proof of Theorem~\ref{main2} in the case $n=3$ (and hence the proof of Theorem~\ref{main} when $n=3$)
is based on a mixture of methods from geometric measure theory (Eilenberg's inequality),
ideas behind the proof of the Freudenthal suspension theorem \cite{hatcher} and the generalized Hopf invariant from \cite{HST}.
As explained above, Theorem~\ref{main2} proves the first part of Theorem~\ref{main}.

The second part of Theorem~\ref{main}, i.e.\;the case of dimensions $n\geq 4$, is also a consequence of a more general result;
recall that $\pi_n(\Sph^{n-1}) = \bbbz_2$ for $n\geq 4$ (see e.g.\;\cite{hatcher}).
\begin{theorem}
\label{main3}
If $k+1\leq m<2k-1$ and $\pi_m(\Sph^k)\neq 0$, then there is a mapping $f\in C^1(\Sph^{m+1},\Sph^{k+1})$
that is not homotopic to a constant map and such that $\rank df\leq k$ everywhere.
\end{theorem}

The proof of Theorem~\ref{main3} is based on
a beautiful and surprising construction of
Wenger and Young \cite[Theorem~2]{wengery}, who proved that {\em if $k+1\leq m<2k-1$
and $g:\Sph^m\to\Sph^k$ is Lipschitz continuous, then there is a Lipschitz
extension $G:\overbar{\bbbb}^{m+1}\to\bbbr^{k+1}$ such that $\rank dG\leq k$ almost everywhere.}
Since we are interested in $C^1$ mappings rather than Lipschitz ones, we have to modify their construction to make sure that we can find
a $C^1$ extension $G$ when $g$ is $C^1$. Our construction is explicit, while the arguments in \cite{wengery} are based on homotopy theory.

The article is organized in the following way.
In Section~\ref{FST} we recall well known facts related to
suspension and the Freudenthal suspension theorem.
This material will be needed in the proofs of Theorem~\ref{main2} (for $n=3$) and of Theorem~\ref{main3}.
In Section~\ref{hopf} we discuss the generalized Hopf invariant introduced in \cite{HST} and we end the section with the proof of Theorem~\ref{fromHST}, which easily follows from the properties of the generalized Hopf invariant.
The generalized Hopf invariant will also be used in the proof of Theorem~\ref{main2} for $n=3$.
Recall that Theorem~\ref{fromHST} implies Theorems~\ref{main} and~\ref{main2} for $n=2$.
In Section~\ref{sec4} we prove Theorem~\ref{main2} for $n=3$.
This completes the proofs of Theorems~\ref{main} and~\ref{main2} for $n=2,3$.
In the final Section~\ref{sec5} we prove Theorem~\ref{main3}, which implies Theorem~\ref{main} for $n\geq 4$.

Notation used in the article is pretty standard. By $\bbbb^\ell$ will always denote {\em open} balls while the symbol $B$ can be used to denote open or closed balls.
The hemispheres $\Sph^n_{\pm}$ will always be closed. By a smooth mapping we will always mean a $C^\infty$ smooth one.
By a smooth diffeomorphism defined on a closed domain $\overbar{\Omega}$
we mean a diffeomorphism that smoothly extends to a diffeomorphism in a larger domain that contains $\overbar{\Omega}$.

\noindent
{\bf Acknowledgements.}
We thank the referees for a careful reading of the paper and for helpful remarks.
In particular, we are grateful to a referee for pointing out the work of Larry Guth.

\section{The Freudenthal suspension theorem}
\label{FST}

With a continuous map $f:\Sph^n\to\Sph^k$ we can associate the {\em suspension map} $Sf:\Sph^{n+1}\to\Sph^{k+1}$,
which maps $n$-spheres parallel to the equator to the corresponding $k$-spheres parallel to the equator. On each of such spheres the map $Sf$ is a scaled copy of
$f$.

Some basic and easy to verify properties of the suspension map are listed in the next three lemmata.

\begin{lemma}
\label{susp1}
If the maps $f,g:\Sph^n\to\Sph^k$ are homotopic, then their suspensions $Sf,Sg:\Sph^{n+1}\to\Sph^{k+1}$ are homotopic as well.
\end{lemma}
The homotopy between
$Sf$ and $Sg$ is simply the suspension of the homotopy between $f$ and $g$.
\begin{lemma}
\label{susp3}
If a map $f:\Sph^n\to\Sph^k$ is homotopic to a constant map, then its suspension $Sf:\Sph^{n+1}\to\Sph^{k+1}$ is homotopic to a constant map.
\end{lemma}
Indeed, since $f$ is homotopic to a constant map $f_o$, $Sf$ is homotopic to $Sf_o$ (Lemma~\ref{susp1}), but the image of $Sf_o$ is a single meridian in $\Sph^{k+1}$ which is contractible, so $Sf_o$ (and hence $Sf$) is homotopic to a constant map.
\begin{lemma}
\label{susp2}
If $F:\Sph^{n+1}\to\Sph^{k+1}$ maps the equator $\Sph^n\subset\Sph^{n+1}$ to the equator $\Sph^k\subset\Sph^{k+1}$,
the upper hemisphere $\Sph^{n+1}_{+}$ to the upper hemisphere $\Sph^{k+1}_{+}$ and
the lower hemisphere $\Sph^{n+1}_{-}$ to the lower hemisphere $\Sph^{k+1}_{-}$, then $F$ is homotopic to the
suspension $Sf$ of the mapping $f=F|_{\Sph^n}:\Sph^n\to\Sph^k$ between the equators.
\end{lemma}
The homotopy is defined as the continuous family of mappings $F_t$, $0\leq t\leq 1$, such that
on each $n$-sphere parallel to the equator whose vertical distance to the equator is no larger than $t$, the mapping $F_t$ coincides with $Sf$ and on
polar caps consisting of points with the vertical distance to the equator at least $t$, the mapping $F_t$ is a scaled version of the mapping $F$ on the
hemispheres. Then $F_0=F$ and $F_1=Sf$.

The next result is the celebrated Freudenthal suspension theorem \cite[Corollary~4.24]{hatcher}.
\begin{lemma}
\label{freud}
If $n\leq 2k-1$, then every map $F:\Sph^{n+1}\to\Sph^{k+1}$ is homotopic to the suspension $Sf$ of a map $f:\Sph^{n}\to\Sph^k$.
If in addition $n<2k-1$, then $Sf:\Sph^{n+1}\to\Sph^{k+1}$ is homotopic to a constant map if an only if
$f:\Sph^{n}\to\Sph^k$ is homotopic to a constant map.
\end{lemma}
It follows from Lemma~\ref{susp3} that if $f$ is homotopic to a constant map, then $Sf$ is homotopic to a constant map.
However, if $n=2k-1$, it may happen that $Sf$ is homotopic to a constant map, even though $f$ is not.

Some ideas from the proof of the Freudenthal theorem are also used in Section~\ref{sec4}. Actually, the ideas from Section~\ref{sec4}
have been used in \cite{EH} to find a somewhat new proof of the Freudenthal theorem (Lemma~\ref{freud}) that uses only elementary methods from differential topology.

The usual statement of the Freudenthal theorem is that the reduced suspension homomorphism
$$
\Sigma:\pi_n(\Sph^k)\to \pi_{n+1}(\Sph^{k+1})
$$
is an epimorphism for $n\leq 2k-1$ and an isomorphism for $n<2k-1$. However, we do not need to use the reduced suspension
in the article, merely the version stated in Lemma~\ref{freud}.

It is important to note that even if $f$ is smooth, the suspension map $Sf$ is not smooth at the north and south poles.
For example, as observed above,
the suspension $Sf_o$ of a constant map $f_o:\Sph^n\to\Sph^k$ maps $\Sph^{n+1}$ into one meridian in $\Sph^{k+1}$.
If we go along a great circle in $\Sph^{n+1}$ that passes through poles at a constant speed, then in the image
of $Sf_o$ we will go back and forth along one meridian, suddenly changing the direction of the constant speed at the poles,
showing that the derivative of $Sf_o$ is discontinuous at the poles. The discontinuity of the derivative  of the
suspension will cause some technical problems in the proof of Theorem~\ref{main3}. However, we can easily correct the suspension to a smooth mapping.
If we parameterize the hemispheres $\Sph^{n+1}_{\pm}$ and $\Sph_{\pm}^{k+1}$ as graphs over the balls
$\overbar{\bbbb}^{n+1}$ and $\overbar{\bbbb}^{k+1}$, then
in these coordinate systems the suspension $Sf:\Sph^{n+1}_{\pm}\to\Sph^{k+1}_{\pm}$ restricted to the hemispheres becomes
$$
\Phi(x)=|x|f\left(\frac{x}{|x|}\right).
$$
The mapping $\Phi$ has discontinuous derivative at the origin, which corresponds to discontinuity of the derivative of $Sf$ at the poles.
If $\lambda_\eps:[0,1]\to [0,1]$ is a smooth and non-decreasing function such that $\lambda_\eps(t)=0$ on $[0,\eps]$ and
$\lambda_\eps(t)=t$ on $[1-\eps,1]$, then the mapping
$$
\Phi_\eps(x)=\lambda_\eps(|x|)f\left(\frac{x}{|x|}\right)
$$
is smooth and it coincides with $\Phi$ near the boundary of $\overbar{\bbbb}^{n+1}$.
The mapping $\Phi_\eps$ induces a smooth mapping $S_\eps f:\Sph^{n+1}\to\Sph^{k+1}$ that is homotopic to $Sf$ and coincides with $Sf$ in a
neighborhood of the equator.

\section{The generalized Hopf invariant}
\label{hopf}

For a smooth map $f:\Sph^{4n-1}\to\Sph^{2n}$, the classical {\em Hopf invariant} is defined as follows (see \cite{BT}).
Let $\alpha_o$ be the volume form on $\Sph^{2n}$ with
$\int_{\Sph^{2n}}\alpha_o=1$. Then $df^*\alpha_o = f^*d\alpha_o =0$. Since the de Rham cohomology $H^{2n}(\Sph^{4n-1})=0$ is trivial,
there is a smooth $(2n-1)$-form $\omega$ on $\Sph^{4n-1}$ such that
$f^*\alpha_o=d\omega$ and {\em the Hopf invariant of $f$} is defined by
$$
\H f=\int_{\Sph^{4n-1}} \omega\wedge d\omega.
$$
The Hopf invariant is invariant under homotopies (\cite[Proposition~17.22]{BT}),
so it can be defined for any continuous map $f:\Sph^{4n-1}\to\Sph^{2n}$. However, it is no longer defined
by the above formula if $f$ is not sufficiently smooth.

\begin{lemma}
\label{l3}
The Hopf invariant is a non-zero group homomorphism $\H:\pi_{4n-1}(\Sph^{2n})\to\bbbz$ and it is an isomorphism when $n=1$.
\end{lemma}
\begin{remark}
\label{r1}
For the proof that $\H$ is a group homomorphism, see \cite[Proposition~4B.1]{hatcher}.
Hopf \cite[Satz II, Satz II']{Hopf} proved that for any $n$, there is a map
$h:\Sph^{4n-1}\to\Sph^{2n}$ with $\H h\neq 0$ and hence the homomorphism
$\H:\pi_{4n-1}(\Sph^{2n})\to\bbbz$ is non-zero.
Since the Hopf invariant of the Hopf fibration $h:\Sph^3\to\Sph^2$ equals $1$ (\cite[Example~17.23]{BT}),
$\H:\pi_3(\Sph^2)\to\bbbz$ is an isomorphism. However, for $n\geq 2$ the Hopf invariant is never an isomorphism.
Indeed, Adams \cite{adams} proved that mappings with Hopf invariant equal $1$ exist only when $n=1,2$ and $4$,
so these are the only cases when one may suspect $\H$ to be an isomorphism, but
$\pi_7(\Sph^4)=\bbbz\times\bbbz_{12}$ and  $\pi_{15}(\Sph^8)=\bbbz\times\bbbz_{120}$, so $\H$ cannot be an isomorphism.
\end{remark}

Let $f:\Sph^{4n-1}\to\bbbr^m$, $m\geq 2n+1$, be a Lipschitz map such that $\rank df\leq 2n$ almost everywhere.
Let $\alpha$ be any $C^\infty$-smooth $2n$-form on $\bbbr^{m}$. Following \cite{HST} we define a generalized Hopf invariant $\H_\alpha f$ as
described below.

According to Lemma~5.4 in \cite{HST}, the form $f^*\alpha\in L^\infty(\bigwedge^{2n}\Sph^{4n-1})$ is weakly closed.
Since the $L^2$-de Rham cohomology of $\Sph^{4n-1}$ in dimension $2n$ is zero (\cite[Proposition~4.5]{HST}), there is a Sobolev
form $\omega\in W^{1,2}(\bigwedge^{2n-1}\Sph^{4n-1})$ such that $d\omega=f^*\alpha$, and we define
$$
\H_\alpha f = \int_{\Sph^{4n-1}}\omega\wedge d\omega.
$$
The main properties of $\H_\alpha$ are described in the following results (see Propositions~5.5 and~5.8 in \cite{HST}).
\begin{lemma}
\label{l1}
If $\omega_1,\omega_2\in W^{1,2}(\bigwedge^{2n-1}\Sph^{4n-1})$ and $d\omega_1=d\omega_2$ a.e., then
the forms $\omega_i\wedge d\omega_i$, $i=1,2$, are integrable and
$$
\int_{\Sph^{4n-1}} \omega_1\wedge d\omega_1 =
\int_{\Sph^{4n-1}} \omega_2\wedge d\omega_2.
$$
In particular, the definition of $\H_\alpha f$ does not depend on the choice of the form $\omega$.
\end{lemma}
\begin{lemma}
\label{l2}
Let $f,g:\Sph^{4n-1}\to\bbbr^m$, $m\geq 2n+1$, be Lipschitz mappings such that
$\rank df\leq 2n$ and $\rank dg\leq 2n$
almost everywhere, and let $\alpha$ be a smooth $2n$-form on $\bbbr^{m}$.
If $H:[0,1]\times \Sph^{4n-1}\to\bbbr^m$ is a Lipschitz homotopy from $f$ to $g$ that satisfies
$\rank dH\leq 2n$ almost everywhere, then $\H_\alpha f=\H_\alpha g$.
\end{lemma}
That means the generalized Hopf invariant $\H_\alpha f$ is invariant under homotopies whose rank of the derivative does not exceed $2n$.

If $\alpha$ is a smooth $2n$-form on $\bbbr^{2n+1}$ whose restriction to $\Sph^{2n}$ coincides with the fixed volume form $\alpha_o$
and $f:\Sph^{4n-1}\to\Sph^{2n}\subset\bbbr^{2n+1}$ is a Lipschitz map, then $\rank df\leq 2n$ almost everywhere and hence the generalized Hopf invariant
$\H_{\alpha} f$ is well defined.
\begin{corollary}
\label{c1}
If $f:\Sph^{4n-1}\to\Sph^{2n}$ is Lipschitz continuous, then $\H_{\alpha} f = \H f$.
\end{corollary}
\begin{proof}
If $g:\Sph^{4n-1}\to\Sph^{2n}$ is smooth, then $\H_{\alpha}g=\H g$, because in that case the definition of $\H_{\alpha} g$ is identical with the
classical definition of the Hopf invariant. If $g$ is homotopic to $f$, then there is also a Lipschitz homotopy
$H:[0,1]\times\Sph^{4n-1}\to\Sph^{2n}$ between $f$ and $g$ (by a standard approximation argument). Since $H$ takes values into $\Sph^{2n}$,
$\rank dH\leq 2n$ a.e. and hence Lemma~\ref{l2} yields
$$
\H_{\alpha} f=\H_{\alpha} g = \H g=\H f,
$$
where the last equality follows from the homotopy invariance of the classical Hopf invariant.
\end{proof}

The next result is essentially contained in \cite[Theorem~1.7]{HST}.

\begin{corollary}
\label{fromHST2}
If $f:\Sph^{4n-1}\to\Sph^{2n}$ is Lipschitz continuous with $\H f\neq 0$ and $F:\overbar{\bbbb}^{4n}\to\bbbr^{2n+1}$ is a Lipschitz extension of $f$,
then $\rank dF=2n+1$ on a set of positive $4n$-dimensional measure.
\end{corollary}
\begin{proof}
Let $f$ and $F$ be as in the statement. Suppose to the contrary that $\rank dF\leq 2n$ almost everywhere.
Then the Lipschitz homotopy
$$
H(t,\theta):[0,1]\times\Sph^{4n-1}\to\bbbr^{2n+1},
\qquad
H(t,\theta)=F(t\theta)
$$
from the constant map $g(\theta)=F(0)$ to $f$ satisfies $\rank dH\leq 2n$ almost everywhere. Since the mappings $f,g, H$ satisfy assumptions of
Lemma~\ref{l2}, we have that Lemma~\ref{l2} together with Corollary~\ref{c1} yield
$$
\H f = \H_{\alpha} f = \H_{\alpha} g =0,
$$
which is a contradiction.
\end{proof}

Theorem~\ref{fromHST} is now a straightforward consequence of Corollary~\ref{c1}.

\begin{proof}[Proof of Theorem~\ref{fromHST}]
First we will prove that $\rank df=2n$ on a set of positive measure. Suppose to the contrary that $\rank df<2n$ almost everywhere.
Let $\alpha$ be as in Corollary~\ref{c1}. Then $f^*\alpha=0$, so $\H_{\alpha} f=0$ and hence $\H f=\H_{\alpha}f=0$, which is a contradiction.
This proves that the set $A$ where $df$ exists and satisfies $\rank df=2n$ has positive measure. It remains to prove that the set $f(A)$ is dense in $\Sph^{2n}$.
Suppose to the contrary that $f(A)\cap \bbbb(y_o,\eps)=\varnothing$ for some ball $\bbbb(y_o,\eps)\subset\Sph^{2n}$.
Then stretching along meridians with $y_o$ regarded as the north pole we can find a Lipschitz homotopy between $f$ and a mapping $f_1$ which maps $A$ to the south pole.
Hence $\rank df_1<2n$ almost everywhere, so by the first part of the proof $\H f_1=0$ and therefore
$0\neq \H f=\H f_1=0$, which is a contradiction.
\end{proof}

\section{Proof of Theorem~\ref{main2}}
\label{sec4}

Recall that Theorem~\ref{fromHST} completes the proof of Theorem~\ref{main2} (and hence that of Theorem~\ref{main}) when $n=2$. Thus we can assume that $n=3$.

Let $f:\Sph^4\to \Sph^3$ be a Lipschitz map that is not homotopic to a constant map.
Let $A$ be the set of points where $f$ is differentiable and $\rank df=3$. We need to prove that $A$ has positive measure and that
its image $f(A)$ is dense in $\Sph^3$. Once we prove that the set $A$ has positive measure, the fact that the set $f(A)$ is dense in $\Sph^3$
will follow from the same argument as the one
in the last step in the proof of Theorem~\ref{fromHST}. Thus it remains to prove that the measure of $A$ is positive.
Assume, to the contrary, that $\rank df\leq 2$ almost everywhere.

The idea is to show that $f$ can be deformed to a Lipschitz map with the rank of the derivative less than or equal to $2$, that maps the equator
$\Sph^3$ of $\Sph^4$ to the equator $\Sph^2$ of $\Sph^3$ and hemispheres to hemispheres. Since the map $f$ is not homotopic to the constant map,
it easily follows from Lemmata~\ref{susp3} and~\ref{susp2}
that the map between the equators
is not homotopic to a constant map either. Hence its Hopf invariant is not zero (Lemma~\ref{l3}).
However, it follows now from Corollary~\ref{fromHST2} that the rank of the derivative of its
extension to the upper (or lower) hemisphere has to be equal $3$ on a set of positive measure, which is a contradiction.

The idea of deformation of the map to a map
that sends the equator to the equator and hemispheres to the hemispheres is closely related to the proof of the Freudenthal suspension theorem
(i.e., Lemma~\ref{freud}).
Indeed, this and Lemma~\ref{susp2} imply that
$f:\Sph^{4}\to\Sph^3$ is
homotopic to the suspension of a map between equators, which is a part of the statement of Freudenthal's theorem.
However, we cannot use the Freudenthal theorem in the proof of Theorem~\ref{main2} directly, because the homotopy between $f$ and the suspension map
may possibly increase the rank of the derivative, i.e.
the homotopic suspension map may have the rank of the derivative equal $3$ on a set of positive measure and we will not obtain any contradiction.

In the first step of the construction of the deformation we need to find two points $y_1,y_2\in\Sph^3$ with `small' pre-images $f^{-1}(y_1)$, $f^{-1}(y_2)$.
To do this we use the next lemma.

\begin{lemma}
\label{eilen}
If $f:\M^m\to \N^n$ is a Lipschitz map between closed Riemannian manifolds of dimensions $m$ and $n$ respectively, then $\HH^{m-n}(f^{-1}(y))<\infty$ for a.e. $y\in \N^n$.
\end{lemma}
Here $\HH^{m-n}$ stands for the Hausdorff measure. This lemma is a direct consequence of Eilenberg's inequality~\cite[Theorem~13.3.1]{buragoz}.
In particular, for almost all $y\in\Sph^3$, $\HH^{1}(f^{-1}(y))<\infty$ and we would like to conclude that for almost all $y_1,y_2\in\Sph^3$,
$\HH^2(f^{-1}(y_1)\times f^{-1}(y_2))<\infty$. However, it cannot be directly concluded from the estimate for the Hausdorff measure of the factors. In fact, the Hausdorff
dimension of the Cartesian product of compact sets $A,B$ can be larger than the sum of Hausdorff dimensions of the sets $A$ and $B$: Theorem~5.11 in \cite{falconer}
provides an example of compact sets $A,B\subset\bbbr$, each of Hausdorff dimension zero, and such that $\HH^1(A\times B)>0$.
Fortunately, a small trick allows us to show that  $\HH^2(f^{-1}(y_1)\times f^{-1}(y_2))<\infty$ for almost all $y_1,y_2\in\Sph^3$ as a direct consequence of Lemma~\ref{eilen}:
Since the map
$$
F:\Sph^4\times\Sph^4\to\Sph^3\times\Sph^3,
\qquad
F(x_1,x_2)=(f(x_1),f(x_2))
$$
is Lipschitz continuous, it follows from Lemma~\ref{eilen} that
$$
\HH^2(f^{-1}(y_1)\times f^{-1}(y_2)) = \HH^2(F^{-1}(y_1,y_2))<\infty
$$
for almost all $y_1,y_2\in\Sph^3$.

Choose such points $y_1,y_2\in\Sph^3$, $y_1\neq y_2$. The sets $f^{-1}(y_1)$ and $f^{-1}(y_2)$ are compact and disjoint.
We want to show that there is a diffeomorphism of $\Sph^4$ that moves one of the
sets to a small neighborhood of a north pole and the other one to a small neighborhood of a south pole.
To construct such a diffeomorphism it will be easier to work in $\bbbr^4$ rather than with $\Sph^4$, but that can be easily achieved.
Let $z\in\Sph^4\setminus (f^{-1}(y_1)\cup f^{-1}(y_2))$  and consider the stereographic projection from
$\Sph^4$ onto $\bbbr^4$ with $z$ as a north pole. With a slight abuse of notation we will denote by $f^{-1}(y_1)$ and $f^{-1}(y_2)$ the corresponding (compact and disjoint) images in $\bbbr^4$.

Consider the map
\begin{equation}
\label{eq2}
\pi: f^{-1}(y_1)\times f^{-1}(y_2)\to\bbbr\bbbp^3
\end{equation}
which assigns to any pair of points $x_1\in f^{-1}(y_1)$ and $x_2\in f^{-1}(y_2)$ the line passing through $x_1$ and $x_2$. It is easy to see that $\pi$ is Lipschitz,
because the distance between the sets $f^{-1}(y_1)$ and $f^{-1}(y_2)$ is positive. Since
$\HH^2(f^{-1}(y_1)\times f^{-1}(y_2))<\infty$, the set
$$
\pi(f^{-1}(y_1)\times f^{-1}(y_2))\subset\bbbr\bbbp^3
$$
is compact and has finite two-dimensional Hausdorff measure.
Since the space $\bbbr\bbbp^3$ is three dimensional, the mapping \eqref{eq2} is not surjective and we can find
$$
v\in \bbbr\bbbp^3\setminus\pi(f^{-1}(y_1)\times f^{-1}(y_2)).
$$
This simply means that the lines parallel to $v$ and passing though the points of $f^{-1}(y_1)$ do not
intersect $f^{-1}(y_2)$. Denote the union of all lines parallel to $v$ and passing through $f^{-1}(y_1)$ by $V$.
Note that $V$ is closed, so there is an open set $U$ containing $V$, whose closure is disjoint from
the compact set $f^{-1}(y_2)$. The direction $v$ defines a vector field on $V$ which can be extended to a bounded and smooth vector field with support contained in $U$.
The flow of this vector field defines a one-parameter group of diffeomorphisms of $\bbbr^4$ which moves $f^{-1}(y_1)$ arbitrarily far away and does not move $f^{-1}(y_2)$.
Making a small adjustment so that the vector field vanishes near the infinity we may transform it back to $\Sph^4$ through the stereographic projection.
As a consequence we find a diffeomorphism of $\Sph^4$ which does not move $f^{-1}(y_2)$ but it moves $f^{-1}(y_1)$ arbitrarily close to the north pole.
Now we can find a one-parameter group of diffeomorphisms that moves points along meridians towards the south pole but does not move the image of $f^{-1}(y_1)$ that is already near the north pole. This family of diffeomorphisms moves $f^{-1}(y_2)$ to the interior of the closed lower hemisphere $\Sph^4_-$ of $\Sph^4$.
This simply means that there is a diffeomorphism $\Phi:\Sph^4\to\Sph^4$ such that $\Phi(f^{-1}(y_1))\subset{\rm int}\,\Sph^4_+$
and $\Phi(f^{-1}(y_2))\subset{\rm int}\,\Sph^4_-$.
Observe that $\Phi$ is homotopic to the identity, because $\Phi$ is constructed through two one-parameter groups of diffeomorphisms connecting $\Phi$ to the identity
(or simply because any orientation preserving diffeomorphism of $\Sph^4$ is homotopic to the identity, as a mapping of degree $1$). Hence
$f_1=f\circ\Phi^{-1}$ is homotopic to $f$. Let $\Sph=\Sph^4_+\cap\Sph^4_-$ be the equator of $\Sph^4$.

Note that $y_1\not\in f_1(\Sph^4_-)$ because $f_1^{-1}(y_1)=\Phi(f^{-1}(y_1))\subset {\rm int}\,\Sph^4_+$. Similarly
$y_2\not\in f_1(\Sph^4_+)$. Hence, for a small $\eps>0$,
$$
\bbbb(y_1,\eps)\cap f_1(\Sph^4_-)=\bbbb(y_2,\eps)\cap f_1(\Sph^4_+)=\varnothing.
$$
Let $\Psi_t:\Sph^3\to \Sph^3$ be a continuous family of smooth mappings such that $\Psi_0=\id$ and
$\Psi_1$ retracts $\Sph^3\setminus (\bbbb(y_1,\eps)\cup \bbbb(y_2,\eps))$ onto an equator $\tilde{\Sph}$ of $\Sph^3$ that separates the sets
$\bbbb(y_1,\eps)$ and $\bbbb(y_2,\eps)$. Namely, $\Psi_1$ stretches the balls
$\bbbb(y_{i},\eps)\cap\Sph^3$, for $i=1,2$, in $\Sph^3$ onto the hemispheres $\Sph^3_{\pm}$ of $\Sph^3$, retracting everything what is between these $\eps$-balls onto the equator.

Clearly, $f_1$ (and hence $f$) is homotopic to
$$
f_2=\Psi_1\circ f_1=\Psi_1\circ f\circ\Phi^{-1}.
$$
Note that
\begin{equation}
\label{e1}
f_2(\Sph)\subset\tilde{\Sph},
\quad
f_2(\Sph^4_+)\subset\Sph^3_+,
\quad
\text{and}
\quad
f_2(\Sph^4_-)\subset\Sph^3_-.
\end{equation}
Indeed,
$$
f_1(\Sph^4_+)\subset\Sph^3\setminus \bbbb(y_2,\eps)
\quad
\text{so}
\quad
f_2(\Sph^4_+)\subset \Psi_1(\Sph^3\setminus \bbbb(y_2,\eps))=\Sph^3_+.
$$
Similarly $f_2(\Sph^4_-)\subset \Sph^3_-$ and hence
$$
f_2(\Sph)=f_2(\Sph^4_+\cap \Sph^4_-)\subset\Sph^3_+\cap\Sph^3_-=\tilde{\Sph}.
$$
Note that $\rank df_2=\rank d(\Psi_1\circ f\circ\Phi^{-1})\leq 2$ a.e. by the chain rule.

Since $f_2:\Sph^4\to\Sph^3$ is homotopic to $f$, it is not homotopic to a constant map.
Now Lemma~\ref{susp2} and
\eqref{e1} yield that the mapping $f_2$ is homotopic to the suspension $Sh$ of
the map $h=f_2|_{\Sph}:\Sph\to\tilde{\Sph}$.
Since $f_2$ is not homotopic to a constant map, $Sh$ is not homotopic to a constant map either.
This and Lemma~\ref{susp3} imply that $h:\Sph\to\tilde{\Sph}$ is not homotopic to a constant map.

Since $h$ is a mapping form a $3$-sphere to a $2$-sphere that is not homotopic to a constant map, its Hopf invariant is non-zero (Lemma~\ref{l3}).
Also, $h$ is Lipschitz
and the Lipschitz extension $f_2$ of $h$ maps $\Sph^4_+$ to $\Sph^3_+$, thus it follows from Corollary~\ref{fromHST2} that $df_2$ has rank $3$
on a subset of $\Sph^4_+$ of positive measure, which is a contradiction.
This completes the proof of Theorem~\ref{main2} and hence that of Theorem~\ref{main}
when $n=2,3$.
\hfill $\Box$

Now it remains to prove Theorem~\ref{main3}.

\section{Proof of Theorem~\ref{main3}}
\label{sec5}
\begin{figure}[h!]
\begin{minipage}{0.6\textwidth}
Let $\lambda_{s,r}:\bbbr\to\bbbr$, for $0<s<r<1$, be a smooth, odd, and non-decreasing function such that
$\lambda_{s,r}(t)=1$ when $|t|>r$ and $\lambda(t)=t$ for $|t|<s$.

The smooth mapping $\Lambda:\bbbr^{k+1}\to\bbbr^{k+1}$
\begin{equation}
\label{lambda}
(x_1,x_2,\ldots,x_{k+1})\stackrel{\Lambda}{\longmapsto} (\lambda_{s,r}(x_1),\ldots,\lambda_{s,r}(x_{k+1}))
\end{equation}
maps $\bbbr^{k+1}$
onto the cube $[-1,1]^{k+1}$ in such a way that the interior neighborhood of the boundary $\partial [-1,1]^{k+1}$ is mapped onto the boundary,
and the complement $\bbbr^{k+1}\setminus [-1,1]^{k+1}$ is also smoothly mapped onto $\partial[-1,1]^{k+1}$.
Note, however, that $\Lambda|_{\partial [-1,1]^{k+1}} \ne \id$, and hence $\Lambda$ is not a retraction.
\end{minipage}
\begin{minipage}{0.37\textwidth}

\begin{tikzpicture}[scale=0.8]
\draw[->] (-3,0)--(3,0);
\draw[->] (0,-3)--(0,3);
\foreach \i in {-1,-0.6,-0.3,0.3,0.6,1}
{\draw ({2*\i},0.05)--({2*\i},-0.05);}
\draw (-0.05,2)--(0.05,2);
\draw (-0.05,-2)--(0.05,-2);
\draw (-3,-2) -- (-1.2,-2) to [out=0,in=-135] (-0.6,-0.6) -- (0.6,0.6) to [out=45, in=180] (1.2,2) -- (3,2);
\node at (2,-0.3) {$1$};
\node at (1.2,-0.3) {$r$};
\node at (0.6,-0.3) {$s$};
\node at (-2.1,-0.3) {$-1$};
\node[left] at (0,2) {$1$};
\node[left] at (0,-2) {$-1$};
\node at (-4,0) {};
\end{tikzpicture}

\caption{Function~$\lambda_{s,r}$.}

\end{minipage}

\end{figure}
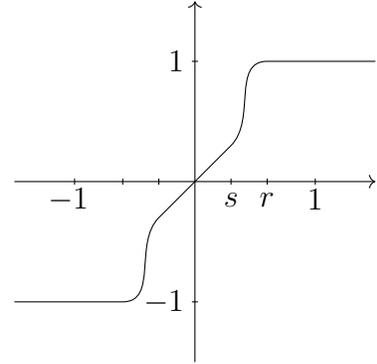
Obviously, $x\mapsto\frac{1}{2}\Lambda(2x)$ has the same properties,
with $[-1/2,1/2]^{k+1}$ in place of $[-1,1]^{k+1}$, and similarly we can rescale and shift this mapping to be used on any other cube in $\bbbr^{k+1}$.

Since $\pi_m(\Sph^k)\neq 0$, there is a $\phi\in C^\infty(\Sph^m,\partial[-\frac{1}{2},\frac{1}{2}]^{k+1})$
that represents a non-trivial element in
$\pi_m(\partial [-\frac{1}{2},\frac{1}{2}]^{k+1})=\pi_m(\Sph^{k})$.
Clearly, we can choose $\phi$ to be continuous, but smoothing $\phi$ and then composing it with a projection onto the boundary of the cube as described
above gives a mapping onto $\partial[-\frac{1}{2},\frac{1}{2}]^{k+1}$ that is $C^\infty$ smooth as a mapping into $\bbbr^{k+1}$.

We reduce the proof of Theorem~\ref{main3} to the following lemma.

\begin{lemma}
\label{enough}
There is a mapping $F\in C^1(\overbar{\bbbb}^{m+1}, [-\frac{1}{2},\frac{1}{2}]^{k+1})$
satisfying $\rank dF\leq k$ everywhere such that $F$ maps the boundary $\partial \bbbb^{m+1}=\Sph^m$ to $\partial [-\frac{1}{2},\frac{1}{2}]^{k+1}$ and
$F|_{\partial\bbbb^{m+1}}=\phi$.
\end{lemma}

Before we prove the lemma we show how Theorem~\ref{main3} follows from it.
Once we have a mapping $F$ as above, we glue two copies of this mapping along the common boundary $\partial \bbbb^{m+1}=\Sph^m$.
We obtain a mapping into two copies of $[-\frac{1}{2},\frac{1}{2}]^{k+1}$ glued along the common boundary
$\partial [-\frac{1}{2},\frac{1}{2}]^{k+1}\approx\Sph^k$, so we essentially obtain a mapping into $\Sph^{k+1}$.

To do it in a smooth way and to have $\Sph^{k+1}$ as the target, let $\xi:[0,\infty)\to [0,\infty)$ be a smooth function satisfying $\xi(t)\leq 1/t$ for all $t>0$, $\xi(t)=1$ for
$t\in[0,1/8]$ and $\xi(t)=1/t$ for $t\geq 1/4$.
Then $\Xi^\ell:\bbbr^{\ell}\to\bbbr^{\ell}$, $\Xi^\ell(x)=\xi(|x|) x$,
is smooth and maps $\bbbr^{\ell}$ to $\overbar{\bbbb}^{\ell}$ and maps $\partial[-\frac{1}{2},\frac{1}{2}]^{\ell}$
onto $\partial\bbbb^{\ell}$.

Since the composition does not increase the rank of the derivative, the derivative of the function
$\tilde{f}=\Xi^{k+1}\circ F\circ\Xi^{m+1}:\overbar{\bbbb}^{m+1}\to\overbar{\bbbb}^{k+1}$
has  rank at most $k$ everywhere and
it is constant along radii near $\partial \bbbb^{m+1}$,
which is guaranteed by $\Xi^{m+1}$. Thus the radial derivative of $\tilde{f}$ vanishes at $\partial \bbbb^{m+1}$,
and $\tilde{f}$ maps $\partial\bbbb^{m+1}=\Sph^m$ onto $\partial\bbbb^{k+1}=\Sph^{k}$.
Let $\Phi_{\pm}:\overbar{\bbbb}^{k+1}\to\Sph^{k+1}_{\pm}$ be diffeomorphisms of
$\overbar{\bbbb}^{k+1}$ onto the closed upper and lower hemispheres that are smooth
up to the boundary and equal to the identity on $\partial\bbbb^{k+1}$.
Then we smoothly glue two copies of $\tilde{f}$, defining  $f:\Sph^{m+1}\to\Sph^{k+1}$ by the formula
$$
f(x_1,\ldots,x_{m+1},x_{m+2})=
\begin{cases}
\Phi_{+}\circ\tilde{f}(x_1,\ldots,x_{m+1})&\text{if $x_{m+2}\geq 0$},\\
\Phi_{-}\circ\tilde{f}(x_1,\ldots,x_{m+1})&\text{if $x_{m+2}\leq 0$}.
\end{cases}
$$
The mapping $f|_{\Sph^m}:\Sph^m\to\Sph^k$, where $\Sph^m\subset\Sph^{m+1}$ and $\Sph^k\subset\Sph^{k+1}$ are equators,
is not homotopic to a constant map, because the mapping $\phi$ is not homotopic to a constant map.
The mapping $f:\Sph^{m+1}\to\Sph^{k+1}$ is homotopic to the suspension of $f|_{\Sph^m}$
(Lemma~\ref{susp2}) and since $m<2k-1$, it is not homotopic to the constant map
(Lemma~\ref{freud}). This completes the proof of Theorem~\ref{main3} and it remains to prove Lemma~\ref{enough}.

\begin{proof}[Proof of Lemma~\ref{enough}]
In what follows, we denote by $\bbbb^\ell$ the unit ball in $\bbbr^\ell$ centered at the origin. We also denote by $\sigma B$ a ball concentric with
$B$ and with radius $\sigma>0$ times that of $B$.

We shall repeatedly use the following geometric facts:
\begin{lemma}
\label{lem:ma}
Let $B_1,\ldots,B_j\subset \bbbb^\ell$ and
$\tB_1, \tB_2,\ldots,\tB_j\subset \bbbb^\ell$ be two families of pairwise
disjoint, closed balls. Then there exists a smooth diffeomorphism
$\Psi:\overbar{\bbbb}^\ell\to\overbar{\bbbb}^\ell$, $\Psi|_{\partial\bbbb^\ell}=\id$,
which maps $B_i$ to $\tB_i$ for $i=1,2,\ldots, j$ in such a way that $\Psi|_{B_i}$ is a translation and scaling.
\end{lemma}
Consider the
cubical ($k+1$)-dimensional complex obtained by partitioning the unit cube $[-\frac{1}{2},\frac{1}{2}]^{k+1}$
into $n^{k+1}$ equal cubes of edge-length $1/n$; denote these cubes by $J_i$, $i=1,2,\ldots,N=n^{k+1}$ and by $S=\bigcup_{i=1}^N\partial J_i$ the $k$-skeleton of the complex.

\begin{lemma}
\label{lem:em}
There exists a smooth mapping $R:\bbbr^{k+1}\to \bbbr^{k+1}$ with the following
properties:
\begin{itemize}
\item $R$ maps a neighborhood of $S$ to $S$:
\begin{itemize}
\item for each cube $J_i$, if $B_i$ is the $(k+1)$-dimensional ball inscribed into $J_i$, then $J_i\setminus \frac{1}{2}B_i$ is mapped onto $\partial J_i$,
\item  $R$ projects  $\bbbr^{k+1}\setminus [-\frac{1}{2},\frac{1}{2}]^{k+1}$ onto $\partial [-\frac{1}{2},\frac{1}{2}]^{k+1}$,
\end{itemize}
\item $R$ is the same, up to translation, in each of the cubes $J_i$,
\item $R$ is homotopic to identity on $\partial J_i$ for each $i$ and on $\partial [-\frac{1}{2},\frac{1}{2}]^{k+1}$.
\end{itemize}
\end{lemma}
\begin{proof}
Observe that the unit ball $\bbbb^{k+1}$ is inscribed in the cube $[-1,1]^{k+1}$.
The mapping $\Lambda$ defined in \eqref{lambda} maps the cube $[-1,1]^{k+1}$
onto itself in such a way that the interior neighborhood of the boundary $\partial [-1,1]^{k+1}$ is mapped onto the boundary.
If we choose $s=\frac{1}{4\sqrt{k+1}}$, $r=2s$, then everything in $[-1,1]^{k+1}$ lying outside the cube $[-\frac{1}{2\sqrt{k+1}},\frac{1}{2\sqrt{k+1}}]^{k+1}$,
in particular $[-1,1]^{k+1}\setminus \frac{1}{2}\bbbb^{k+1}$, is mapped onto $\partial [-1,1]^{k+1}$.

The mapping is obviously smooth and homotopic to identity on $\partial [-1,1]^{k+1}$.

We can use $\Lambda$ (rescaled and translated) to project an interior neighborhood of the boundary of each
$J_i$ onto that boundary of $J_i$. The resulting mapping
is of class $C^\infty$ on the whole cube $[-\frac{1}{2},\frac{1}{2}]^{k+1}$.
Even the corners of the cubes $J_i$ do not cause any problems, because the entire neighborhood of
each of the corners is mapped into the corner and the mapping is $C^\infty$ there as it is a constant one.
This way $R$ is defined already on $[-\frac{1}{2},\frac{1}{2}]^{k+1}$.
Finally, for $x$ outside
$[-\frac{1}{2},\frac{1}{2}]^{k+1}$ we have well defined nearest point projection
$\pi:\bbbr^{k+1}\setminus [-\frac{1}{2},\frac{1}{2}]^{k+1}\to \partial [-\frac{1}{2},\frac{1}{2}]^{k+1}$; we take $R(x)=R(\pi(x))$.
Although, $\pi$ is not smooth, it is easy to see that the mapping $R$ is smooth, because in the normal directions
to faces of any dimension, the mapping $R\circ\pi$ is constant in a neighborhood of the point on the edge where we take the normal direction.
\end{proof}

We concentrate now on the construction of $F$.
We begin by choosing $N=n^{k+1}$ disjoint closed balls $\bbbb_i$, of radius $\frac{2}{n}$, all inside $\frac{1}{2}\bbbb^{m+1}$ (see Figure \ref{fig:1}). This is possible if we choose $n$ large enough. Indeed, the ball $\frac{1}{2}\overbar{\bbbb}^{m+1}$ contains a cube of edge length $\frac{1}{\sqrt{m+1}}$, and in it one can fit at least
$\left[\frac{n}{5\sqrt{m+1}}\right]^{m+1}$ cubes of edge length $\frac{5}{n}$ with pairwise disjoint interiors, where $[t]$ is the integer part of $t$.
For $n>(10\sqrt{m+1})^{m+1}$, we have
$$
\left[\frac{n}{5\sqrt{m+1}}\right]^{m+1}> \left(\frac{n}{10\sqrt{m+1}}\right)^{m+1}=n^m \frac{n}{(10\sqrt{m+1})^{m+1}}> n^{k+1}.
$$

\begin{figure}[h!]
\setlength{\currentparindent}{\parindent}
\begin{minipage}{0.6\textwidth}
\setlength{\parskip}{\bigskipamount}
\setlength{\parindent}{12pt}
\bigskip

Finally, in the interior of each of these cubes one can find a closed ball
$\overbar{\bbbb}_i$, of radius $\frac{2}{n}$,
concentric with the cube. Since the balls do not touch the boundaries of the cubes, they are pairwise disjoint.

Let $\overbar{\Omega}_o=\overbar{\bbbb}^{m+1}\setminus \bigcup_{i=1}^N \bbbb_i$. This will be the domain of the initial step of our construction, which will be then iterated inside each of the balls $\bbbb_i$. Note that
$\overbar{\Omega}_0$ is an $(m+1)$-manifold with $N+1$ boundary components, all of which are $m$-spheres.

\end{minipage}
\begin{minipage}{0.39\textwidth}
\begin{tikzpicture}[scale=0.5]
\filldraw[fill=lightgray!20!white] (6,0) circle (4);
\draw[darkgray] (2,0) arc [x radius=4, y radius =2, start angle =-180, end angle=0];
\filldraw[draw=gray,fill=lightgray!40!white] (6,0) circle (2);

\filldraw[ball color=white] (5.2,0.8) circle (0.5);
\filldraw[fill=white,opacity=0.5] (5.2,0.8) circle (0.5);
\draw (5.2,0.8) circle (0.5);
\draw[darkgray] (4.7,0.8) arc [x radius=0.5, y radius =0.25, start angle =-180, end angle=0];
\filldraw[ball color=white] (6.8,0.8) circle (0.5);
\filldraw[fill=white,opacity=0.5] (6.8,0.8) circle (0.5);
\draw (6.8,0.8) circle (0.5);
\draw[darkgray] (6.3,0.8) arc [x radius=0.5, y radius =0.25, start angle =-180, end angle=0];
\filldraw[ball color=white] (5.2,-0.8) circle (0.5);
\filldraw[fill=white,opacity=0.5] (5.2,-0.8) circle (0.5);
\draw (5.2,-0.8) circle (0.5);

\draw[darkgray] (4.7,-0.8) arc [x radius=0.5, y radius =0.25, start angle =-180, end angle=0];
\filldraw[ball color=white] (6.8,-0.8) circle (0.5);
\filldraw[fill=white,opacity=0.5] (6.8,-0.8) circle (0.5);
\draw (6.8,-0.8) circle (0.5);
\draw[darkgray] (6.3,-0.8) arc [x radius=0.5, y radius =0.25, start angle =-180, end angle=0];
\draw[lightgray!50!gray] (4,0) arc [x radius=2, y radius =1, start angle =-180, end angle=0];
\node at (0,0) {};
\end{tikzpicture}
\captionsetup{width=.8\linewidth}
\caption{Disjoint~balls~$\overbar{\bbbb}_i$ of radius $\frac{2}{n}$ lie inside~$\frac{1}{2}\bbbb^{m+1}$.}
\label{fig:1}
\end{minipage}

\end{figure}

Using Lemma~\ref{lem:ma}, we find a diffeomorphism $G_1:\overbar{\bbbb}^{m+1}\to\overbar{\bbbb}^{m+1}$ such that
\begin{itemize}
\item $G_1$ is identity on $\partial \bbbb^{m+1}$,
\item it maps the balls $\overbar{\bbbb}_i$ into $N$ identical
closed balls $\hat{K}_i$, arranged along the $x_{m+1}$ axis, mapping these balls by a translation and scaling,
see Figure \ref{fig:2}.
\end{itemize}

Since $n$ is large, the balls $\hat{K}_i$ might have to be smaller than the balls $\overbar{\bbbb}_i$.

\begin{figure}[ht]
\begin{tikzpicture}[scale=0.9]
\begin{scope}[scale=0.6]
\node[scale=1.2] at (9.5,-3.8) {$\bbbb^{m+1}$};
\filldraw[fill=lightgray!20!white] (6,0) circle (4);
\draw[darkgray] (2,0) arc [x radius=4, y radius =2, start angle =-180, end angle=0];
\filldraw[draw=gray,fill=lightgray!40!white] (6,0) circle (2);

\filldraw[ball color=white] (5.2,0.8) circle (0.5);
\filldraw[fill=white,opacity=0.5] (5.2,0.8) circle (0.5);
\draw (5.2,0.8) circle (0.5);
\draw[darkgray] (4.7,0.8) arc [x radius=0.5, y radius =0.25, start angle =-180, end angle=0];
\filldraw[ball color=white] (6.8,0.8) circle (0.5);
\filldraw[fill=white,opacity=0.5] (6.8,0.8) circle (0.5);
\draw (6.8,0.8) circle (0.5);
\draw[darkgray] (6.3,0.8) arc [x radius=0.5, y radius =0.25, start angle =-180, end angle=0];
\filldraw[ball color=white] (5.2,-0.8) circle (0.5);
\filldraw[fill=white,opacity=0.5] (5.2,-0.8) circle (0.5);
\draw (5.2,-0.8) circle (0.5);

\draw[darkgray] (4.7,-0.8) arc [x radius=0.5, y radius =0.25, start angle =-180, end angle=0];
\filldraw[ball color=white] (6.8,-0.8) circle (0.5);
\filldraw[fill=white,opacity=0.5] (6.8,-0.8) circle (0.5);
\draw (6.8,-0.8) circle (0.5);
\draw[darkgray] (6.3,-0.8) arc [x radius=0.5, y radius =0.25, start angle =-180, end angle=0];

\draw[lightgray!50!gray] (4,0) arc [x radius=2, y radius =1, start angle =-180, end angle=0];
\end{scope}
\begin{scope}[shift={(8,0)},scale=0.6]
\node[scale=1.2] at (9.5,-3.8) {$\bbbb^{m+1}$};
\filldraw[fill=lightgray!20!white] (6,0) circle (4);

\draw[dashed,thick] (6,5)--(6,3.6);
\draw[dashed, gray] (6,3.6)--(6,-4);
\draw[dashed,thick] (6,-4)--(6,-5);
\filldraw (6,3.7) circle (0.05);
\filldraw[lightgray] (6,-3.8) circle (0.05);
\foreach \i in {0,1,2,3}
{
\begin{scope}[shift={(0,-1.5*\i)}]
\filldraw[ball color=white] (6,2.3) circle (0.5);
\filldraw[fill=white,opacity=0.5] (6,2.3) circle (0.5);
\draw (6,2.3) circle (0.5);
\draw[gray] (5.5,2.3) arc [x radius=0.5, y radius =0.25, start angle =-180, end angle=0];
\end{scope}
}

\draw[darkgray] (2,0) arc [x radius=4, y radius =2, start angle =-180, end angle=0];
\end{scope}
\draw[thick,->] (6.5,0)--(8.5,0);
\node[above] at (7.5,0) {$G_1$};

\draw[->] (4.1,0.6) to [out=30, in =150] (11.6,1.5);
\end{tikzpicture}
\caption{The diffeomorphism $G_1$ rearranges the balls $\overbar{\bbbb}_i$, possibly shrinking them,
so that their centers lie on the $x_{m+1}$ axis. It maps $\overbar{\bbbb}_i$ to $\hat{K}_i$ by a similarity (scaling+translation) transformation.}
\label{fig:2}
\end{figure}
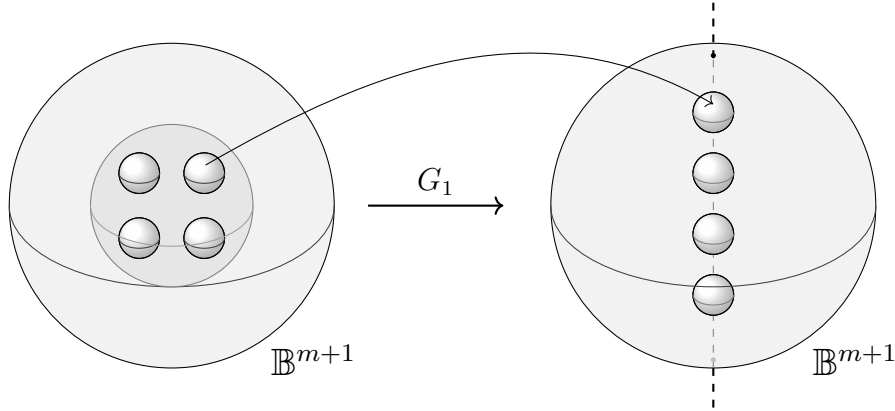

For the next step in the construction we will need the following lemma.
\begin{lemma}
\label{joasia}
Let $\phi\in C^\infty(\Sph^m,\partial[-\frac{1}{2},\frac{1}{2}]^{k+1})$ be as in Lemma~\ref{enough}. Then there is a smooth map
$h\in C^\infty(\Sph^{m-1},\Sph^{k-1})$, not homotopic to a constant map, such that for any map
$P\colon\Sph^k\to\partial[-\frac{1}{2},\frac{1}{2}]^{k+1}$ that is homotopic to the radial projection
$\pi\colon\Sph^k\to\partial[-\frac{1}{2},\frac{1}{2}]^{k+1}$, the map
$P\circ Sh:\Sph^m\to\partial[-\frac{1}{2},\frac{1}{2}]^{k+1}$ is homotopic to $\phi$.
\end{lemma}
\begin{remark}
\label{r2}
Recall the operation of smooth suspension $S_\eps$ has been defined at the end of Section~\ref{FST} and it follows from Lemma~\ref{joasia} that
$P\circ S_\eps h$ is homotopic to $\phi$.
\end{remark}
\begin{proof}
Since $\pi$ is a homeomorphism, the map $\tilde{\phi}=\pi^{-1}\circ\phi:\Sph^m\to\Sph^k$ is well defined and not homotopic to a constant map. Next,
$m\leq 2k-2$, thus it follows from Lemma~\ref{freud} that $\tilde{\phi}$ is homotopic to the suspension $Sh$ of a map
$h\in C^\infty(\Sph^{m-1},\Sph^{k-1})$. Hence $\pi\circ Sh$ is homotopic to $\phi$ and thus $P\circ Sh$ is homotopic to $\phi$ for any map $P$ as in the statement of the lemma. The map $h$ is not homotopic to a constant map by Lemma~\ref{susp3}.
\end{proof}

We extend $h$ radially to the mapping
$$
\bbbr^m\ni x\
\stackrel{{H_1}}{\longmapsto} \
|x|\, h\left(\frac{x}{|x|}\right)\in\bbbr^k,
$$
so each sphere (centered at the origin) of radius $r$ is mapped to the sphere of radius $r$ by a scaled version of the mapping $h$.
Then we extend ${H_1}$ to the mapping
$$
\bbbr^{m+1}\ni (x,t)
\ \stackrel{H_2}{\longmapsto} \
\left(|x|\, h\left(\frac{x}{|x|}\right),t\right)\in\bbbr^{k+1}.
$$
The mapping $H_2$ maps the ($m+1$)-balls of radius $r$ centered at the $t$-axis (i.e. $x_{m+1}$ axis) to
($k+1$)-balls of radius $r$ centered at the $t$-axis (i.e. $x_{k+1}$ axis). Thus it maps $\overbar{\bbbb}^{m+1}$
onto $\overbar{\bbbb}^{k+1}$ and each of the balls $\hat{K}_i$ to a corresponding ball $K_i$ in
$\bbbr^{k+1}$.

Moreover, since $H_1$ is a scaled version of $h$ on each of the spheres centered at the origin, it follows that the restriction of $H_2$ to the boundaries of the balls
\begin{equation}
\label{B52}
H_2:\partial\bbbb^{m+1}\to\partial\bbbb^{k+1}
\quad
\text{and}
\quad
H_2:\partial \hat{K}_i\to\partial K_i
\quad
\text{for}\ i=1,\ldots, N,
\end{equation}
is the same mapping $Sh:\Sph^{m}\to\Sph^{k}$ (homotopic to $\tilde{\phi}$), up to a similarity in source and target.

In particular we have
$$
H_2:\overbar{\bbbb}^{m+1}\setminus\bigcup_{i=1}^N\textrm{int}\,{\hat{K}_i}\to
\overbar{\bbbb}^{k+1}\setminus \bigcup_{i=1}^N \textrm{int}\,K_i
$$
and we will consider the mapping $H_2$ restricted to that set only.

As explained at the end of Section~\ref{FST}, the mapping $H_1$ is not smooth at the origin and hence $H_2$ is not smooth along the $x_{m+1}$-axis.
In particular, the restrictions of $H_2$ in
\eqref{B52} are not smooth at the poles of the spheres.
However, the mappings \eqref{B52} are homotopic to the smooth suspension $S_\eps h$ discussed in Section~\ref{FST}.
Therefore we may modify $H_2$ to obtain a mapping that coincides with a scaled version of $S_\eps h$ on each of the spheres
$\partial\bbbb^{m+1}$ and $\partial\hat{K}_i$ and is smooth in a neighborhood of each of the spheres.
The resulting mapping is still not smooth on a compact subset of the $x_{m+1}$-axis that is in the interior of the set
$\overbar{\bbbb}^{m+1}\setminus\bigcup_{i=1}^N{\hat{K}_i}$ and hence it does not touch the spheres.
A standard mollification argument allows us to smooth it out and finally we obtain a smooth map
$$
H:\overbar{\bbbb}^{m+1}\setminus\bigcup_{i=1}^N\textrm{int}\,{\hat{K}_i}\to
\overbar{\bbbb}^{k+1}\setminus \bigcup_{i=1}^N \textrm{int}\,K_i
$$
that coincides with a scaled version of $S_\eps h$ on each of the spheres $\partial\bbbb^{m+1}$ and $\partial\hat{K}_i$
(see Figure \ref{fig:4}).

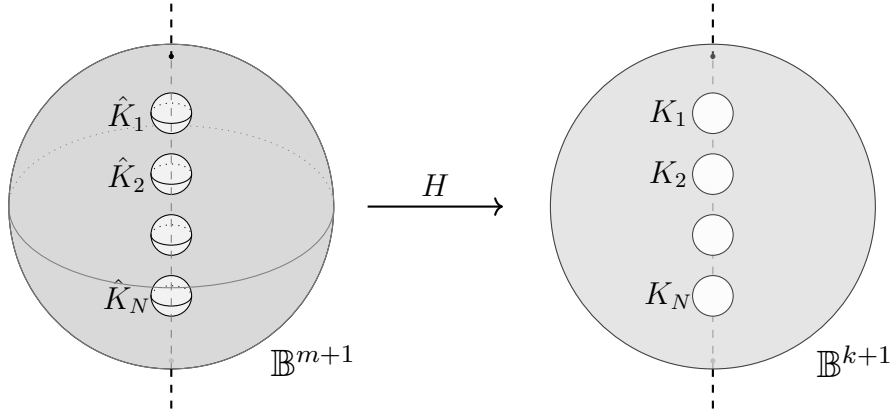
\begin{figure}[ht]
\begin{tikzpicture}[scale=0.9]
\begin{scope}[scale=0.6]
\node[scale=1.2] at (9.5,-3.8) {$\bbbb^{m+1}$};

\filldraw[fill=lightgray!60!white] (6,0) circle (4);
\draw[gray] (6,0) circle (4);

\draw[dotted,gray] (2,0) arc [x radius=4, y radius =2, start angle =180, end angle=0];
\foreach \i in {1,2,3,4}
{
\begin{scope}[shift={(0,-1.5*(\i-1))}]
\filldraw[fill= lightgray!20!white] (6,2.3) circle (0.5);
\draw (5.5,2.3) arc [x radius=0.5, y radius =0.25, start angle =-180, end angle=0];
\draw[dotted] (5.5,2.3) arc [x radius=0.5, y radius =0.25, start angle =180, end angle=0];
\end{scope}
}
\draw[dashed,thick] (6,5)--(6,3.6);
\draw[dashed, gray] (6,3.6)--(6,-4);
\draw[dashed,thick] (6,-4)--(6,-5);
\filldraw (6,3.7) circle (0.05);
\filldraw[lightgray] (6,-3.8) circle (0.05);

\node at (4.9,2.3) {$\hat K_1$};
\node at (4.9,0.8) {$\hat K_2$};
\node at (4.9,-2.2) {$\hat K_N$};
\draw[gray] (2,0) arc [x radius=4, y radius =2, start angle =-180, end angle=0];
\end{scope}
\draw[thick,->] (6.5,0)--(8.5,0);
\node[above] at (7.5,0) {$H$};

\begin{scope}[shift={(8,0)}, scale=0.6]
\node[scale=1.2] at (9.5,-3.8) {$\bbbb^{k+1}$};
\draw[dashed,thick] (6,5)--(6,3.6);
\draw[dashed, gray] (6,3.6)--(6,-4);
\draw[dashed,thick] (6,-4)--(6,-5);
\filldraw (6,3.7) circle (0.05);
\filldraw[lightgray] (6,-3.8) circle (0.05);
\filldraw[color= lightgray,opacity=0.4] (6,0) circle (4);
\draw[darkgray] (6,0) circle (4);
\foreach \i in {1,2,3,4}
{
\begin{scope}[shift={(0,-1.5*(\i-1))}]
\filldraw[draw=darkgray,fill=white!95!lightgray] (6,2.3) circle (0.5);
\end{scope}
}
\node at (4.9,2.3) {$K_1$};
\node at (4.9,0.8) {$K_2$};
\node at (4.9,-2.2) {$K_N$};

\end{scope}
\end{tikzpicture}
\caption{The smooth mapping $H$ maps
the spheres $\partial\bbbb^{m+1}$ and $\partial\hat{K}_i$ onto $\partial\bbbb^{k+1}$ and
$\partial K_i$ by a scaled copy of $S_\eps h$.}
\label{fig:4}
\end{figure}

Let the unit cube $Q=[-\frac{1}{2},\frac{1}{2}]^{k+1}\subset \bbbr^{k+1}$ be divided into an even grid of $N=n^{k+1}$ cubes $J_i$, of
edge length $1/n$.

Note that $Q\subset\frac{1}{2}\sqrt{k+1}\, \overbar{\bbbb}^{k+1}$.

In the next step we again use Lemma~\ref{lem:ma} to find a diffeomorphism $G_2$ that maps
$\overbar{\bbbb}^{k+1}$ to
$\frac{1}{2}\sqrt{k+1}\overbar{\bbbb}^{k+1}$ in such a way that
\begin{itemize}
\item $G_2$ maps $\partial \bbbb^{k+1}$ to $\partial(\frac{1}{2}\sqrt{k+1}\bbbb^{k+1})$ by similarity (in fact, scaling),
\item $G_2$ maps each of the balls $K_i$ into a ball $L_i$ such that the ball $\frac{11}{10} L_i$ is inscribed into the cube $J_i$, and $G_2|_{K_i}$ is a similarity (translation+scaling).
\end{itemize}
The diffeomorphism $G_2$ is depicted in Figure \ref{fig:5}.

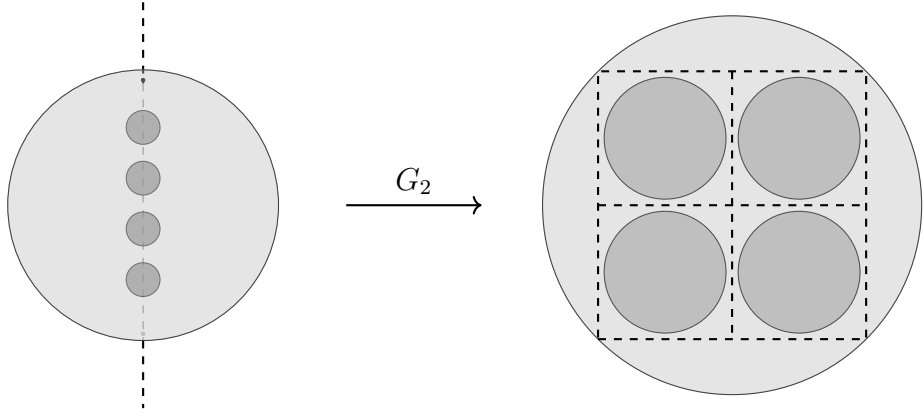
\begin{figure}[ht]
\begin{tikzpicture}[scale=0.9]

\draw[thick,->] (6,0)--(8,0);
\node[above] at (7,0) {$G_2$};

\begin{scope}[scale=0.5]

\draw[dashed,thick] (6,6)--(6,3.6);
\draw[dashed, gray] (6,3.6)--(6,-4);
\draw[dashed,thick] (6,-4)--(6,-6);
\filldraw (6,3.7) circle (0.05);
\filldraw[lightgray] (6,-3.8) circle (0.05);
\foreach \i in {0,1,2,3}
{
\begin{scope}[shift={(0,-1.5*\i)}]
\filldraw[draw=black,fill= gray,opacity=0.7] (6,2.3) circle (0.5);
\end{scope}
}

\filldraw[color= lightgray,opacity=0.4] (6,0) circle (4);
\draw[darkgray] (6,0) circle (4);
\end{scope}

\begin{scope}[shift={(7.5,0)}, scale=0.7]


\foreach \i in {-1,1}
{
\foreach \j in {-1,1}
{
\begin{scope}[shift={({sin(45)*2*\i} ,{sin(45)*2*\j})}]
\filldraw[draw=darkgray!10!black,fill= lightgray,opacity=1] (6,0) circle ({10/11*2*sin(45)});
\end{scope}
}
}

\filldraw[color= lightgray,opacity=0.4] (6,0) circle (4);
\draw[darkgray] (6,0) circle (4);
\draw[thick,dashed] ({6+4*cos(45)},{4*sin(45)})--({6+4*cos(-45)},{4*sin(-45)})--({6+4*cos(-135)},{4*sin(-135)})--({6+4*cos(135)},{4*sin(135)})--cycle;
\draw[thick,dashed] (6,{4*sin(45)})--(6,{-4*sin(45)});
\draw[thick,dashed] ({6-4*sin(45)},0)--({6+4*sin(45)},0);

\end{scope}
\end{tikzpicture}
\caption{The diffeomorphism $G_2$ rearranges the balls $K_i$ in $\bbbb^{k+1}$, mapping $\partial\bbbb^{k+1}$ by scaling to a ball of radius $\frac{1}{2}\sqrt{k+1}$ and balls $K_i$ to balls $L_i$, almost inscribed into a grid obtained by partitioning the unit cube $[-\frac{1}{2},\frac{1}{2}]^{k+1}$ into $N=n^{k+1}$ cubes of edge length $\frac{1}{n}$. }
\label{fig:5}
\end{figure}

Finally, we use the mapping $R$ defined in Lemma~\ref{lem:em} to project
$\frac{1}{2}\sqrt{k+1}\overbar{\bbbb}^{k+1}\setminus \bigcup_{i=1}^N L_i$ onto the $k$ dimensional complex
$S=\bigcup_i\partial J_i$, see Figure \ref{fig:6}.

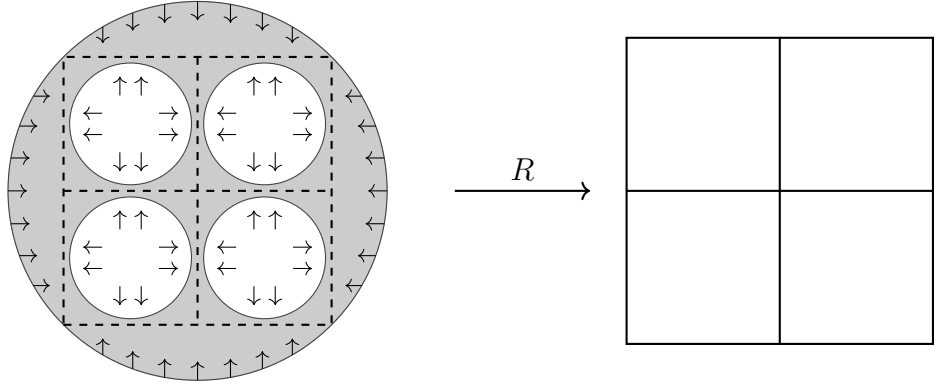
\begin{figure}[ht]
\begin{tikzpicture}[scale=0.9]

\begin{scope}[scale=0.7]


\filldraw[color= gray,opacity=0.4] (6,0) circle (4);
\draw[darkgray] (6,0) circle (4);
\draw[thick,dashed] ({6+4*cos(45)},{4*sin(45)})--({6+4*cos(-45)},{4*sin(-45)})--({6+4*cos(-135)},{4*sin(-135)})--({6+4*cos(135)},{4*sin(135)})--cycle;
\draw[thick,dashed] (6,{4*sin(45)})--(6,{-4*sin(45)});
\draw[thick,dashed] ({6-4*sin(45)},0)--({6+4*sin(45)},0);
\foreach \i in {-1,1}
{
\foreach \j in {-1,1}
{
\begin{scope}[shift={({sin(45)*2*\i} ,{sin(45)*2*\j})}]
\filldraw[draw=darkgray,fill= white] (6,0) circle ({10/11*2*sin(45)});
\foreach \t in {-0.75,0.75}
{
\draw[->] ({6+0.3*\t},{0.6})--({6+0.3*\t},1);
}

\foreach \t in {-0.75,0.75}
{
\draw[->] ({6+0.3*\t},{-0.6})--({6+0.3*\t},-1);
}
\foreach \t in {-0.75,0.75}
{
\draw[->] (5.4,{0.3*\t})--(5,{0.3*\t});
}

\foreach \t in {-0.75,0.75}
{
\draw[->] (6.6,{0.3*\t})--(7,{0.3*\t});
}
\end{scope}
}
}
\foreach \t in {60,70,...,120}
{
\draw[->] ({6+4*cos(\t)},{4*sin(\t)})--({6+4*cos(\t)},{3.6*sin(\t)});
}

\foreach \t in {-30,-20,...,30}
{
\draw[->] ({6+4*cos(\t)},{4*sin(\t)})--({6+3.6*cos(\t)},{4*sin(\t)});
}
\foreach \t in {-60,-70,...,-120}
{
\draw[->] ({6+4*cos(\t)},{4*sin(\t)})--({6+4*cos(\t)},{3.6*sin(\t)});
}
\foreach \t in {150,160,...,210}
{
\draw[->] ({6+4*cos(\t)},{4*sin(\t)})--({6+3.6*cos(\t)},{4*sin(\t)});
}
\end{scope}
\begin{scope}[shift={(8,0)},scale=0.8]

\filldraw[thick, fill=white] ({6+4*cos(45)},{4*sin(45)})--({6+4*cos(-45)},{4*sin(-45)})--({6+4*cos(-135)},{4*sin(-135)})--({6+4*cos(135)},{4*sin(135)})--cycle;
\draw[thick] (6,{4*sin(45)})--(6,{-4*sin(45)});
\draw[thick] ({6-4*sin(45)},0)--({6+4*sin(45)},0);
\end{scope}
\draw[thick,->] (8,0)--(10,0);
\node[above] at (9,0) {$R$};
\end{tikzpicture}
\caption{ In the final step we use the mapping $R$ to project the image of $\overbar{\Omega}_o$, here in darker shade, onto the $k$-complex $S$.}
\label{fig:6}
\end{figure}

Let $\hat F=R\circ G_2\circ H \circ G_1:\overbar{\Omega}_o=\overbar{\bbbb}^{m+1}\setminus \bigcup_{i=1}^N \bbbb_i \to S$.

On the boundary of each ball $\overbar{\bbbb}_i$ the mapping $\hat F|_{\partial\bbbb_i}\to\partial J_i$ is, up to a similarity in source and image,
identical with some fixed mapping $g:\Sph^m\to \partial [-\frac{1}{2},\frac{1}{2}]^{k+1}$,
that is homotopic to $\phi$ by Lemma~\ref{joasia} and Remark~\ref{r2}, see Figure~\ref{fig:7}.
Similarly, $\hat F|_{\partial \bbbb^{m+1}}:\partial \bbbb^{m+1}\to\partial [-\frac{1}{2},\frac{1}{2}]^{k+1} $
is homotopic to $\phi$ (but not necessarily equal, up to scaling, to $g$).
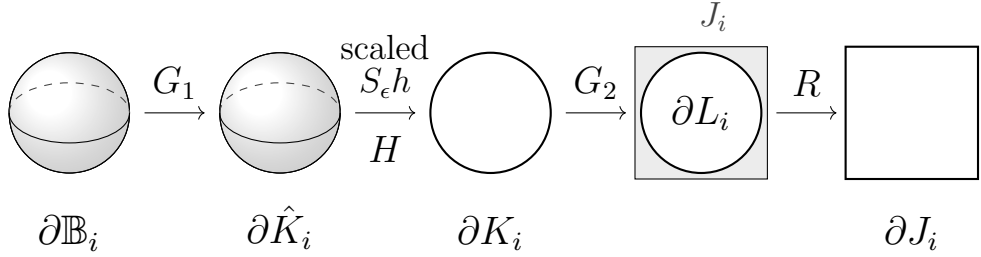
\begin{figure}[ht]
\begin{tikzpicture}[scale=0.8]

\foreach \i in {0,1}
{
\begin{scope}[shift={({3.5*\i},0)}]
\filldraw[shading=ball, ball color= white] (1,0) circle (1);
\filldraw[fill= white, opacity=0.7] (1,0) circle (1);
\draw (1,0) circle (1);
\draw (0,0) arc [x radius=1, y radius =0.5, start angle=-180, end angle=0];
\draw[darkgray, dashed] (0,0) arc [x radius=1, y radius =0.5, start angle=180, end angle=0];
\end{scope}
}
\filldraw[thick,fill=white,draw=black] (8,0) circle (1);
\filldraw[fill=lightgray!30!white,draw=] (10.4,1.1) -- (12.6,1.1) -- (12.6,-1.1) -- (10.4,-1.1) -- cycle;
\filldraw[thick,fill=white,draw=black] (11.5,0) circle (1);
\draw[thick] (13.9,1.1) -- (16.1,1.1) -- (16.1,-1.1) -- (13.9,-1.1) -- cycle;
\node[scale=1.3] at  (1,-2) {$\partial\bbbb_i$};
\node[scale=1.25] at  (4.5,-1.9) {$\partial\hat K_i$};
\node[scale=1.25] at  (8,-2) {$\partial K_i$};
\node[darkgray,scale=1.1] at  (11.7,1.6) {$J_i$};
\node[scale=1.25] at  (11.5,0) {$\partial L_i$};
\node[scale=1.25] at  (15,-2) {$\partial J_i$};

\draw[->] (2.25,0)--(3.25,0);
\node[scale=1.2] at (2.75,0.5) {$G_1$};
\draw[->] (5.75,0)--(6.75,0);
\node[scale=1.1] at (6.25,1.1) {scaled};
\node[scale=1.1] at (6.25,0.5) {$S_\epsilon h$};
\node[scale=1.2] at (6.25,-0.6) {$H$};
\draw[->] (9.25,0)--(10.25,0);
\node[scale=1.2] at (9.75,0.5) {$G_2$};
\draw[->] (12.75,0)--(13.75,0);
\node[scale=1.2] at (13.25,0.5) {$R$};
\end{tikzpicture}
\caption{The fate of $\partial \bbbb_i$ throughout the construction.}
\label{fig:7}
\end{figure}

Since we want to iterate the construction, by gluing into each of $\overbar{\bbbb}_i$ a rescaled copy of the mapping $\hat F$,
we want to ensure that the maps indeed glue in a $C^\infty$ manner (although $C^1$ would be enough).
To this end, we want to have $\hat F|_{\partial \bbbb^{m+1}}$ and $\hat F|_{\partial \bbbb_i}$ equal, up to scaling,
to the mapping $\phi$ (given in the statement of Lemma~\ref{enough}).
Moreover, for the maps to glue in a $C^\infty$ manner we want the map $\hat F$,
in a neighborhood of the boundary of $\overbar{\Omega}_o=\overbar{\bbbb}^{m+1}\setminus \bigcup_i\bbbb_i$,
to be constant in the normal directions to the boundary of $\Omega_o$.

At the moment, $\hat F|_{\partial \bbbb^{m+1}}$ and $\hat F|_{\partial \bbbb_i}$ are only homotopic to $\phi$.
Let us thus correct $\hat F$ in three steps in the following way.
\begin{itemize}
\item[-]
First, we modify $\hat{F}$ to a smooth map $\hat{F}_1$ which coincides with $\hat{F}$ on
$\frac{2}{3}\overbar{\bbbb}^{m+1}\setminus\bigcup_{i=1}^N\bbbb_i$, and
$\hat{F}_1|_{\partial\bbbb^{m+1}}$ equals $\phi$.
This is possible, because $\hat{F}_1|_{\partial(\frac{2}{3}\bbbb^{m+1})}$ and $\phi$ are smoothly homotopic
(up to the scaling that identifies
$\partial(\frac{2}{3}\bbbb^{m+1})$ with $\partial\bbbb^{m+1}$) as mappings into $\partial[-\frac{1}{2},\frac{1}{2}]^{k+1}$.
\item[-] Next, we want to ensure that the mapping is constant along radii on
$\overbar{\bbbb}^{m+1}\setminus\frac{3}{4}\bbbb^{m+1}$
(which is a smaller annulus than $\overbar{\bbbb}^{m+1}\setminus\frac{2}{3}\bbbb^{m+1}$). We can do it by pre-composing $\hat{F}_1$ with the mapping
$$
\overbar{\bbbb}^{m+1}\ni x
\stackrel{\Phi_1}{\longmapsto}
\lambda_{\frac{2}{3},\frac{3}{4}}(|x|) \frac{x}{|x|}\in\overbar{\bbbb}^{m+1},
$$
where the function $\lambda_{s,t}$ is as defined at the beginning of the proof.
The mapping $\hat{F}_2=\hat{F}_1\circ\Phi_1$ is constant along the radii on $\bbbb^{m+1}\setminus\frac{3}{4}\bbbb^{m+1}$ and $\hat{F}_2$
equals $\phi$ on $\partial\bbbb^{m+1}$.
\item[-] Using the same argument as above we can modify $\hat{F}_2$ in a small neighborhood of each of the boundaries $\partial\bbbb_i$ so that the resulting mapping $F_o$ equals $\phi$ (up to scaling) on each of the spheres $\partial\bbbb_i$, is constant in normal directions near $\partial\bbbb_i$ and
maps $\partial\bbbb_i$ onto $\partial J_i$.
\end{itemize}

The mapping $F_o:\overbar{\Omega}_o\to S\subset [-\frac{1}{2},\frac{1}{2}]^{k+1}$ is our initial step of the construction.

To inductively fill the map $F_o$ into the holes $\bbbb_i$ (up to a Cantor set),
we associate to each $\overbar{\bbbb}_i$ its center $x_i$ and the similarity map
$$
\sigma_i \colon \overbar{\bbbb}^{m+1} \to \overbar{\bbbb}_i,
\qquad
\sigma_i(x)=\frac{2}{n}x + x_i;
$$
recall that the radius of $\bbbb_i$ equals $\frac{2}{n}$.

Each mapping $\sigma_i$ maps $\overbar{\Omega}_o$ into the ball $\overbar{\bbbb}_i$, so that if
$$
\overbar{\Omega}_1=\bigcup_{i=1}^N\sigma_i(\overbar{\Omega}_o)
\quad
\text{and}
\quad
D_1=\overbar{\Omega}_o\cup\overbar{\Omega}_1,
$$
then the set $D_1$ is obtained by adding to
$\overbar{\Omega}_o$ scaled copies of $\overbar{\Omega}_o$ inside each of the holes
$\bbbb_i$.
The set $\overbar{\Omega}_o$ has $n^{k+1}$ holes, each of radius $\frac{2}{n}$ while $D_1$ has
$(n^{k+1})^2$ holes, each of radius $(\frac{2}{n})^2$.

We define now inductively
$$
\overbar{\Omega}_\ell=\bigcup_{i=1}^N \sigma_i(\overbar{\Omega}_{\ell-1}),
\qquad
D_\ell = \bigcup_{j=0}^\ell \overbar{\Omega}_j
$$
so $D_\ell$ has $(n^{k+1})^{\ell+1}$ holes, each of radius $(\frac{2}{n})^{\ell+1}$.

Let $D=\bigcup_{\ell=0}^{\infty} D_\ell$ so $C=\overbar{\bbbb}^{m+1}\setminus D$ is a Cantor set.
The mapping $F_o$ has already been defined and we define inductively
$$
F_\ell:D_\ell\to \left[-\frac{1}{2},\frac{1}{2}\right]^{k+1}
$$
for each $\ell\geq 1$ by
$$
F_\ell|_{D_{\ell-1}}=F_{\ell-1},
\quad
F_\ell|_{\sigma_i(\Omega_{\ell-1})}=\tau_i\circ F_{\ell-1}\circ\sigma_i^{-1},
\quad
i=1,2,\ldots,N,
$$
where
$$
\tau_i:\left[-\frac{1}{2},\frac{1}{2}\right]^{k+1}\to J_i\ \
\stackrel{\text{isometric}}{\approx} \ \ \left[-\frac{1}{2n},\frac{1}{2n}\right]^{k+1}
$$
is the translation and scaling transformation.

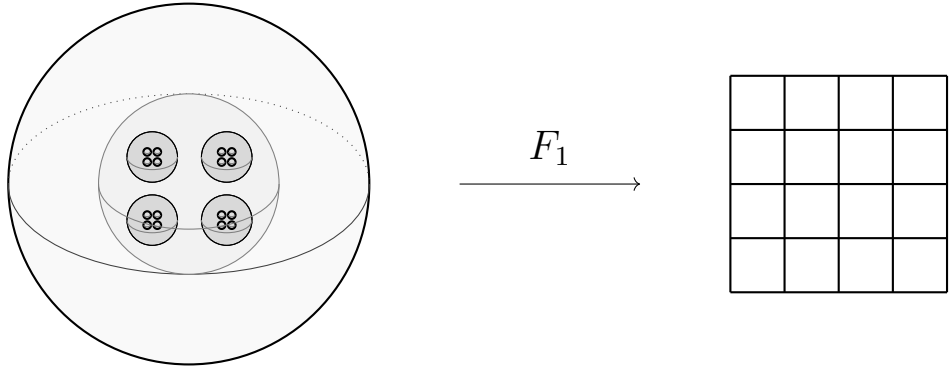
\begin{figure}
\begin{tikzpicture}[scale=0.6]
\filldraw[thick,fill= lightgray!10!white] (6,0) circle (4);
\draw[darkgray] (2,0) arc [x radius=4, y radius=2, start angle=-180, end angle=0];
\draw[darkgray, dotted] (2,0) arc [x radius=4, y radius=2, start angle=180, end angle=0];
\filldraw[fill= lightgray!20!white,draw=gray] (6,0) circle (2);

\begin{scope}[shift={(4.35,0.6)},scale=0.14]

\filldraw[fill=lightgray!60!white] (6,0) circle (2);
\draw[gray] (4,0) arc [x radius=2, y radius =1, start angle =-180, end angle=0];
\filldraw[fill=lightgray!60!white] (6,0) circle (4);
\draw (6,0) circle (4);
\draw[gray] (2,0) arc [x radius=4, y radius =2, start angle =-180, end angle=0];
\filldraw[thick,fill=lightgray!50!white] (5.2,0.8) circle (0.6);
\draw[gray] (4.7,0.8) arc [x radius=0.5, y radius =0.25, start angle =-180, end angle=0];
\filldraw[thick,fill=lightgray!50!white] (6.8,0.8) circle (0.6);
\draw[gray] (6.3,0.8) arc [x radius=0.5, y radius =0.25, start angle =-180, end angle=0];
\filldraw[thick,fill=lightgray!50!white] (5.2,-0.8) circle (0.6);
\draw[gray] (4.7,-0.8) arc [x radius=0.5, y radius =0.25, start angle =-180, end angle=0];
\filldraw[thick,fill=lightgray!50!white] (6.8,-0.8) circle (0.6);
\draw[gray] (6.3,-0.8) arc [x radius=0.5, y radius =0.25, start angle =-180, end angle=0];
\end{scope}
\begin{scope}[shift={(6,0.6)},scale=0.14]

\filldraw[fill=lightgray!60!white] (6,0) circle (2);
\draw[lightgray!50!gray] (4,0) arc [x radius=2, y radius =1, start angle =-180, end angle=0];
\filldraw[fill=lightgray!60!white] (6,0) circle (4);
\draw (6,0) circle (4);
\draw[gray] (2,0) arc [x radius=4, y radius =2, start angle =-180, end angle=0];
\filldraw[thick,fill=lightgray!50!white] (5.2,0.8) circle (0.6);
\draw[gray] (4.7,0.8) arc [x radius=0.5, y radius =0.25, start angle =-180, end angle=0];
\filldraw[thick,fill=lightgray!50!white] (6.8,0.8) circle (0.6);
\draw[gray] (6.3,0.8) arc [x radius=0.5, y radius =0.25, start angle =-180, end angle=0];
\filldraw[thick,fill=lightgray!50!white] (5.2,-0.8) circle (0.6);
\draw[gray] (4.7,-0.8) arc [x radius=0.5, y radius =0.25, start angle =-180, end angle=0];
\filldraw[thick,fill=lightgray!50!white] (6.8,-0.8) circle (0.6);
\draw[gray] (6.3,-0.8) arc [x radius=0.5, y radius =0.25, start angle =-180, end angle=0];
\end{scope}

\begin{scope}[shift={(4.35,-0.8)},scale=0.14]

\filldraw[fill=lightgray!60!white] (6,0) circle (2);
\draw[lightgray!50!gray] (4,0) arc [x radius=2, y radius =1, start angle =-180, end angle=0];
\filldraw[fill=lightgray!60!white] (6,0) circle (4);
\draw (6,0) circle (4);
\draw[gray] (2,0) arc [x radius=4, y radius =2, start angle =-180, end angle=0];
\filldraw[thick,fill=lightgray!50!white] (5.2,0.8) circle (0.6);
\draw[gray] (4.7,0.8) arc [x radius=0.5, y radius =0.25, start angle =-180, end angle=0];
\filldraw[thick,fill=lightgray!50!white] (6.8,0.8) circle (0.6);
\draw[gray] (6.3,0.8) arc [x radius=0.5, y radius =0.25, start angle =-180, end angle=0];
\filldraw[thick,fill=lightgray!50!white] (5.2,-0.8) circle (0.6);
\draw[gray] (4.7,-0.8) arc [x radius=0.5, y radius =0.25, start angle =-180, end angle=0];
\filldraw[thick,fill=lightgray!50!white] (6.8,-0.8) circle (0.6);
\draw[gray] (6.3,-0.8) arc [x radius=0.5, y radius =0.25, start angle =-180, end angle=0];
\end{scope}
\begin{scope}[shift={(6,-0.8)},scale=0.14]

\filldraw[fill=lightgray!60!white] (6,0) circle (2);
\draw[lightgray!50!gray] (4,0) arc [x radius=2, y radius =1, start angle =-180, end angle=0];
\filldraw[fill=lightgray!60!white] (6,0) circle (4);
\draw (6,0) circle (4);
\draw[gray] (2,0) arc [x radius=4, y radius =2, start angle =-180, end angle=0];
\filldraw[thick,fill=lightgray!50!white] (5.2,0.8) circle (0.6);
\draw[gray] (4.7,0.8) arc [x radius=0.5, y radius =0.25, start angle =-180, end angle=0];
\filldraw[thick,fill=lightgray!50!white] (6.8,0.8) circle (0.6);
\draw[gray] (6.3,0.8) arc [x radius=0.5, y radius =0.25, start angle =-180, end angle=0];
\filldraw[thick,fill=lightgray!50!white] (5.2,-0.8) circle (0.6);
\draw[gray] (4.7,-0.8) arc [x radius=0.5, y radius =0.25, start angle =-180, end angle=0];
\filldraw[thick,fill=lightgray!50!white] (6.8,-0.8) circle (0.6);
\draw[gray] (6.3,-0.8) arc [x radius=0.5, y radius =0.25, start angle =-180, end angle=0];
\end{scope}
\draw[lightgray!50!gray,draw=gray] (4,0) arc [x radius=2, y radius =1, start angle =-180, end angle=0];

\begin{scope}[scale=1.2]
\foreach \i in {0,1,...,4}
{
\draw[thick] (15+\i,2)--(15+\i,-2);
\draw[thick] (15,2-\i)--(19,2-\i);
}
\end{scope}

\draw[->] (12,0)--(16,0);
\node[scale=1.3] at (14,0.8) {$F_1$};
\end{tikzpicture}
\caption{The second iteration $F_1$.}
\end{figure}

Now $F:D\to [-\frac{1}{2},\frac{1}{2}]^{k+1}$ is given by $F|_{D_\ell}=F_\ell$ for each $\ell=0,1,2,\ldots$
This map is smooth in $D$ and it continuously extends to the Cantor set $C$. Moreover
$$
\Vert d F_\ell|_{\Omega_\ell}\Vert_\infty =
\frac{1}{n} \Vert d F_{\ell-1}|_{\Omega_{\ell-1}}\Vert_\infty \frac{n}{2}
=\ldots=
\frac{1}{2^\ell} \Vert d F_o|_{\Omega_o}\Vert_\infty \to 0.
$$
Thus $F$ is continuously differentiable also on the Cantor set $C$ with
$DF|_C\equiv 0$.
This completes the proof (and the article).
\end{proof}

\end{document}